\documentclass[11pt]{amsart}
\makeatletter
\def\specialsection{\@startsection{section}{1}%
  \z@{\linespacing\@plus\linespacing}{.5\linespacing}%
  {\normalfont}}
\def\section{\@startsection{section}{1}%
  \z@{.7\linespacing\@plus\linespacing}{.5\linespacing}%
  {\normalfont\scshape}}
\makeatother

\usepackage[a4paper, textwidth=6.5in, textheight = 10.5in, tmargin =2cm, lmargin = 2cm, rmargin=3cm]{geometry}

\usepackage{microtype}

\usepackage{amssymb,amsmath,esint,hyperref,enumitem}
\usepackage{amsfonts}
\usepackage{mathtools}
\usepackage{mathrsfs}
\usepackage{esvect}
\usepackage{amssymb}
\usepackage[all]{xy}

\newtheorem{theorem}{Theorem}[section]
\theoremstyle{plain}
\newtheorem{acknowledgement}{Acknowledgement}

\newtheorem{definition}[theorem]{Definition}

\newtheorem{lemma}[theorem]{Lemma}

\newtheorem{proposition}[theorem]{Proposition}

\numberwithin{equation}{section}
\newcommand{\R}{\mathbb{R}}
\newcommand{\pa}{\partial}

\newcommand{\ld}{\lambda}

\newcommand{\J}{\mathbb{J}}

\begin{document}
\title{Sub-Riemannian curvature of Carnot groups with rank-two distributions}

\author{Isidro H. Munive}
\email{imunive@cimat.mx}

\maketitle

\begin{abstract}
The notion of curvature discussed in this paper is a far going generalization of the Riemannian sectional curvature. It was first introduced by Agrachev, Barilari and Rizzi in \cite{ABR1}, and it is defined for a wide class of optimal control problems: a unified framework including geometric structures such as Riemannian, sub-Riemannian, Finsler, and sub-Finsler structures. In this work we study the \emph{generalized sectional curvature} of  Carnot groups with  rank-two distributions.   In particular, we consider the Cartan group and Carnot groups with horizontal distribution of Goursat-type. In these Carnot groups we characterize ample and equiregular geodesics.  For Carnot groups with horizontal Goursat distribution we show that their \emph{generalized sectional curvatures}     depend only on the Engel part of the distribution. This family of Carnot groups contains naturally the three-dimensional  Heisenberg group, as well as the Engel group.  Moreover, we also show that in the Engel and Cartan  groups   there exist initial covectors for which there is an infinite discrete set of  times at which the corresponding ample geodesics are not equiregular.  
\end{abstract}

\begin{acknowledgement}
The author wishes to thank Andrei Agrachev, Davide Barilari and Luca Rizzi for their constant interest and the many helpful conversations on sub-Riemannian geometry, especially on their paper \cite{ABR1}. The author also would  like to  thank the anonymous reviewers for their constructive comments, which helped to improve the contents of the paper, especially the content of Theorem \ref{NonVanh}. This work was done while the author was visiting the International School for Advanced Studies (SISSA) at Trieste, Italy.    
\end{acknowledgement}

\section{Introduction}

In sub-Riemannian geometry there is no canonical connection such as the miraculous Levi-Civita connection in Riemannian geometry. Despite this obstacle, several groups of mathematicians in recent years have been trying to define a notion of  curvature in the sub-Riemannian setting. The curvature we consider in this paper was  introduced in \cite{ABR1}, and it is a generalization of the \emph{sectional curvature} from Riemannian geometry.

 Let $M$ be an $n-$dimensional Riemannian manifold with Riemannian distance $d$ and let $\gamma_v(t),\gamma_w(s)$ be two arc-length parametrized geodesics with   
\[
\gamma_v(0)=\gamma_w(0)=x_0\in M \quad \text{and}\quad \dot{\gamma}_v(0)=v,\dot{\gamma}_w(0)=w\in T_{x_0}M.  
\]
The geodesic cost associated with $\gamma_v$ is
\[
c_t(x)\doteq-\frac{1}{2t}d^2(x,\gamma_v(t)).
\]
Consider the function 
\[
C(t,s)\doteq -tc_t(\gamma_{w}(s)).
\]

The  \emph{sectional curvature} of a Riemannian manifold can be recover from the asymptotic expansion of the function $C$. It can be shown that  $C$ is smooth at $(0,0)$. The following formula, which is due to Loeper, see \cite{Loeper}, holds true for the Taylor expansion of $C(t,s)$ at $(0,0)$
\begin{eqnarray*}
C(t,s)&=&\frac{1}{2}\left(t^2+s^2-2g(v,w) ts\right)-\frac{1}{6}g\left(R(v,w)v,w\right) t^2s^2\\
&&+t^2s^2o\left(|t|+|s|\right),
\end{eqnarray*}
where $g$ denotes the Riemannian inner product and $R$ is the Riemannian curvature tensor. From the expansion of $C(t,s)$, we easily get $d_{x_0}c_t(\cdot)=g(v,\cdot)$. Now, if we let $\dot{c}_t=\frac{\partial}{\partial t}c_t$, then $d_{x_0}\dot{c}_t=0.$  Hence, the \emph{Hessian} $d^2_{x_0}\dot{c}_t$ is a well-defined quadratic form on $T_{x_0}M$. From the Taylor expansion of $C(t,s)$, we finally get
\begin{equation}
\label{RiemannianEx}
d^2_{x_0}\dot{c}_t(w)=\frac{1}{t^2}g(w,w)+\frac{1}{3}g\left( R(v,w)v,w\right)+O(t).
\end{equation}

The derivative $\dot{c}_t$ of the geodesic cost has a very nice geometric interpretation. Let $W^t_{x,\gamma(t)}\in T_{\gamma(t)}M$ be the final tangent vector of the unique
minimizer connecting $x$ with $\gamma(t)$ in time $t$. From  \cite[Appendix H]{ABR1} we have:
\begin{equation}
\label{dotc}
\dot{c}_t(x) = \frac{1}{2}\|W^t_{x_0,\gamma(t)}-W^t_{x,\gamma(t)}\|^2-\frac{1}{2}\|W^t_{x_0,\gamma(t)}\|^2.
\end{equation}

The ``curvature'' at the initial point $x_0$ is hidden in the behavior of this function for small $t$ and $x$ close to $x_0$. For instance, in a positively (resp. negatively) curved Riemannian manifold, for which Eq. (\ref{dotc}) still holds, then the two tangent vectors, compared at $\gamma(t)$, are more (resp. less) divergent w.r.t. $t$. 

 Despite the  fact that the sub-Riemannian distance has a more complicated behavior than the Riemannian one,  Agrachev, Barilari and Rizzi  wrote in \cite{ABR1} an expansion  for the square of the  sub-Riemannian distance that resembles Eq. (\ref{RiemannianEx}) in order to have a notion of  \emph{curvature} in the sub-Riemannian setting. Indeed, the authors defined \emph{curvature-type invariants} for a wide class of optimal control problems. One must mention here that formula (\ref{dotc}) is also valid in the sub-Riemannian setting, see \cite{ABR1}.

We will now describe some results of the paper \cite{ABR1}. Let $(\mathscr{D},g)$ be a sub-Riemannian structure on a smooth manifold $M$, $\gamma$ be a fixed normal geodesic, with $\gamma(0)=x_0$ and inital covector $\lambda_0$. Let 
\[
c_t(x)\doteq-\frac{1}{2t}d^2(x,\gamma(t)), \quad t>0, 
\]
be the geodesic cost function. If $\gamma$ is also \emph{strongly normal}, see Definition \ref{Strongly}, then
the function $c_t(x)$ is smooth in $t$ and $x$ for small $t>0$ and $x$ close to $x_0$. Moreover, $d_{x_0}c_t=\lambda_0$, and hence, $x_0$ is a critical point for $\dot{c}_t=\pa_tc_t$.

 The \emph{second differential}  of $\dot{c}_t$ is well defined at $x_0$, and we can  associate with the  family of quadratic forms $d^2_{x_0}\dot{c}_t|_{\mathscr{D}_{x_0}}:\mathscr{D}_{x_0}\rightarrow \R$, a family of symmetric operators $\mathcal{Q}_{\lambda_0}(t)$, via the formula
\begin{equation}
d^2_{x_0}\dot{c}_t(v)=g(\mathcal{Q}_{\lambda}(t)v,v),\quad \text{$t>0$ and $v\in\mathscr{D}_{x_0}$}.
\end{equation}

If we also assume that $\gamma$ is \emph{ample}, see Definition \ref{ampledef}, then we have the following Laurent expansion for the family of symmetric operators $\mathcal{Q}_{\lambda_0}(t):\mathscr{D}_{x_0}\rightarrow\mathscr{D}_{x_0}$,
\begin{equation*}
\mathcal{Q}_{\lambda_0}(t)=\frac{1}{t^2}\mathcal{I}_{\lambda_0}+\frac{1}{3}\mathcal{R}_{\lambda_0}+O(t) \quad t>0.
\end{equation*} 

Moreover, in \cite{ABR1} the authors showed that, if  the geodesic $\gamma$ is also \emph{equiregular}, see Definition \ref{ampledef}, the explicit expression of the invariants $\mathcal{I}_{\lambda_0}$ and $\mathcal{R}_{\lambda_0}$  can be computed in terms of the symplectic invariants of the so-called \emph{Jacobi curve}. The invariant $\mathcal{R}_{\lambda_0}$ is called  \emph{the curvature operator}. The Jacobi curve arises naturally from the geometric interpretation of the second derivative of the geodesic cost. These symplectic invariants can be computed via an algorithm which is quite hard to implement. Despite the enormous difficulty of this algorithm  we are able to provide in this paper  the explicit expression of the invariants $\mathcal{I}_{\lambda_0}$ and $\mathcal{R}_{\lambda_0}$ in Carnot groups with horizontal distribution of \emph{Goursat-type}. 

We consider in $\R^n$, $n\geq3$, a system of vector fields $\{X_1,\ldots,X_n\}$ such that
\begin{equation}
\label{GoursatIntro}
\left[X_1,X_2\right]=X_{3}, \quad \left[X_1,X_3\right]=X_{4},\quad \ldots, \quad\left[X_1,X_{n-1}\right]=X_{n},
\end{equation} 
and the other brackets assume to be trivial. Let the $\mathfrak{g}$ be the nilpotent stratified Lie algebra $\mathfrak{g}=\mathfrak{g}_1\oplus\ldots\oplus\mathfrak{g}_{n-1}$, with $\mathfrak{g}_1=\mathrm{span}\ \{X_1,X_2\}$,  $\mathfrak{g}_i=\mathrm{span}\ \{X_{i+1}\}$ for $i=2,\ldots,n-1$. There exists a unique simply connected Lie group $\mathbb{J}^n$ such that $\mathfrak{g}$ is its Lie algebra of left-invariant vector fields. The \emph{Heisenberg group} ($n=3$) and the \emph{Engel group} ($n=4$) are very important examples of this family. The group $\mathbb{J}^n$ is a Carnot group with horizontal distribution of  Goursat-type. 

The name Goursat distribution is related to the work \cite{Goursat}, in which Goursat popularized these distributions. Goursat's predecessors were Engel and Cartan. These distributions are quite important since they are examples of sub-Riemannian manifolds of step higher than two. Goursat distributions of dimension $n$ can be obtained by the \emph{$(n-2)$-fold prolongation} of a two-dimensional surface, see \cite{Mon}. 

A particular physical problem described by the distribution (\ref{GoursatIntro})  is the motion of electric charges in certain static inhomogeneous magnetic fields, see for instance \cite{AnzaldoMonroy}. A different physical problem which is also related to Goursat distributions is the so-called $N-trailer$ problem, which consists of steering a robot with a number $N$ of trailers, see \cite{FLMR}, \cite{Jak}, \cite{JeanF}, \cite{EssaysRobotics} \cite{Laumond}, \cite{Laumond2}, \cite{Li-Canny}, \cite{FLMR1}, \cite{Samson}, \cite{Sordalen}, \cite{TMW} and \cite{TMS}, among many others.

Let us now describe the main results of this paper. For $i=1,\ldots,n$, let $h_i$ be  the linear on fiber function given  by $h_i(\lambda)\doteq\left\langle \ld,X_i\right\rangle$,  with $\lambda\in T^*\R^n$.
\begin{theorem}
\label{gasym}
Let $\gamma:[0,T]\rightarrow\R^n$ be an ample and equiregular geodesic in $\mathbb{J}^n$, $n\geq 3$, with $\gamma(0)=x_0$ and initial covector $\lambda(0)=\lambda_0\in T^*_{x_0}\R^n$. Then,  we have the following explicit expansion of $d^2_{x_0}\dot{c}_t$  as $t\rightarrow 0$, 
\begin{eqnarray}
\label{Expansionct}
d^2_{x_0}\dot{c}_t|_{\mathscr{D}}&=&\frac{1}{t^2}\mathcal{I}_{\lambda_0}+\frac{1}{3}\mathcal{R}_{\lambda_0}+O(t)\\
&=&
\frac{1}{t^2}
\begin{pmatrix}
(n-1)^2 & 0\\
0& 1
\end{pmatrix}
+\frac{1}{3}
\begin{pmatrix}
\frac{3(n-1)}{4(n-1)^2-1}R_{11}(\lambda_0)& 0\\
0 & 0
\end{pmatrix}
+O\left(t\right),\nonumber
\end{eqnarray}
in a suitable orthonormal basis of $\mathscr{D}_{x_0}$,   where
\begin{eqnarray}
\nonumber
R_{11}(\lambda_0)&=&-\frac{1}{6}(n-1)\left(12+n\left(4n-17\right)\right)\left(h^2_3+h_2h_4\right)\\
&&-(n-1)(n-2)(n-3)h^2_3\frac{h^2_2}{h^2_1}.\label{CurRI}
\end{eqnarray}
Here we use the convention $h_4\equiv0$ when $n=3$.
\end{theorem}

In the Heisenberg group all geodesics are  ample and equiregular. The invariant $\mathcal{R}_{\lambda_0}$ in this group, computed for the first time  in the works \cite{ABR1} and \cite{BR1}, is given by
\[
\mathcal{R}_{\lambda_0}\doteq \frac{2}{5}\begin{pmatrix}
h^2_3(\lambda_0)& 0\\
0 & 0
\end{pmatrix}.
\]
For $\mathbb{J}^n$, $n\geq 4$, we obtain from Theorem \ref{gasym} that $\mathcal{R}_{\lambda_0}$  depends only on the \emph{Engel} part of the distribution and on the dimension $n$. Furthermore, notice that contrary to the Heisenberg group case, when $n\geq 4$, the quantity $R_{11}(\lambda_0)$ changes sign and has a singularity when $h_1(\lambda_0)=0$. In Theorem \ref{NonVanh} we show, for instance, that an ample geodesic $\gamma=\pi(\lambda)$ is equiregular at $t$ if and only if $h_1(\lambda(t))\neq 0$. Equiregular geodesics can be thought of as  the microlocal counterpart of  equiregular distributions.

The operator $\mathcal{I}_{\lambda_0}$  carries important geometric information. More precisely, the trace of $\mathcal{I}_{\lambda_0}$ is equal to the \emph{geodesic dimension} of $\mathbb{J}^n$.   The geodesic dimension in $\mathbb{J}^n$, for instance, is defined as follows. Let $\Omega\subset\R^n$ be a bounded and measurable subset of positive volume and let $\Omega_{x_0,t}$, for $0 \leq t \leq1$, be a family of subsets obtained from $\Omega$ by the homothety of $\Omega$ with respect to a fixed point $x_0$ along the shortest geodesics connecting $x_0$ with the points of $\Omega$, so that $\Omega_{x_0,0} = \{x_0\}$, $\Omega_{x_0,1} = \Omega$. The volume of $\Omega_{x_0,t}$ has order $t^{\mathcal{N}_{x0}}$ , where $\mathcal{N}_{x_0}$ is the geodesic dimension at $x_0$ (see \cite[Section 5.6]{ABR1} and also \cite{Ri}  for a more intrinsic and general definition of the geodesic dimension, which is also valid for general metric measure spaces).

The next natural candidate to study would be the \emph{Cartan group} $\mathfrak{C}$, with underlying manifold $\R^5$. The Lie group $\mathfrak{C}$ is the free nilpotent group with five-dimensional nilpotent Lie algebra $\mathfrak{L}=\mathrm{span}\ \{X_1,\ldots, X_5\}$, whose generators satisfy
\[
[X_1,X_2]=X_3, \quad [X_1,X_3]=X_4, \quad [X_2,X_3]=X_5
\]
Notice that the distribution of $\mathfrak{C}$ is not of Goursat-type.

 For the Cartan group we show that the \emph{generalized sectional curvature} is bounded from above by the \emph{energy integral of the pendulum}, which is a first integral of the normal Hamiltonian system. For $i=1,\ldots,5$, let $h_i$ be  the linear on fiber function given  by $h_i(\lambda)\doteq\left\langle \ld,X_i\right\rangle$,  with $\lambda\in T^*\R^5$. Then, the \emph{energy integral of the pendulum} is given by
\[
E=\frac{h^2_3}{2}+h_1h_5-h_2h_4.
\]

\begin{theorem}
\label{casym}
 Let $\gamma:[0,T]\rightarrow\R^5$ be an ample and equiregular geodesic, with $\gamma(0)=x_0$ and initial covector $\lambda(0)=\lambda_0\in T^*_{x_0}\mathbb{R}^5$. We have the following explicit expansion of $d^2_{x_0}\dot{c}_t$  as $t\rightarrow 0$,
\begin{eqnarray*}
d^2_{x_0}\dot{c}_t|_{\mathscr{D}}&=&\frac{1}{t^2}\mathcal{I}_{\lambda_0}+\frac{1}{3}\mathcal{R}_{\lambda_0}+O(t)\\
&=&
\frac{1}{t^2}
\begin{pmatrix}
16 & 0\\
0& 1
\end{pmatrix}
+\frac{1}{3}
\begin{pmatrix}
\frac{4}{21}R_{11}(\lambda_0)& 0\\
0 & 0
\end{pmatrix}
+O\left(t\right),\nonumber
\end{eqnarray*}
in a suitable orthonormal basis of $\mathscr{D}_{x_0}$,  where
\begin{eqnarray}
\label{CurRIE}
R_{11}(\lambda_0)&=&6E-\frac{8}{h^2_3}\left(h_1h_4+h_2h_5\right)^2.
\end{eqnarray}
\end{theorem}
In Proposition \ref{NonVanh3} we  show that for an ample geodesic $\gamma=\pi\circ\lambda$ the times $t$ where $h_3(\lambda(t))=0$ are times of loss of equiregularity.

The function $R_{11}(\lambda)$ in Eq.'s (\ref{CurRI}) and (\ref{CurRIE}) is a symplectic invariant of the so-called \emph{Jacobi curve}. Several authors have used the symplectic invariants associated with a \emph{Jacobi curve}  to get important new results in sub-Riemannian geometry. This set of invariants was first introduced by Agrachev and Gamkrelidze in \cite{AG}, Agrachev and Zelenko in \cite{AZ} and successively extended by Zelenko and Li in \cite{ZL}. In the works \cite{AL} and \cite{AL1}, Agrachev and Lee made extensive use of these invariants to establish for three-dimensional \emph{Sasakian} manifolds a generalized measure-contraction property, from which it follows:  a volume doubling property, a local Poincar\'e inequality, a Harnack inequality for positve harmonic functions  and a Liouville property. They were also able to prove  Bishop and Laplacian comparison theorems on such three-dimensional manifolds. One has to mention here that the \emph{three-dimensional Heisenberg group} is a very important example of a Sasakian manifold. Furthermore, in a recent paper \cite{BR1}, Barilari and Rizzi employed the invariants of the Jacobi curve to establish comparison theorems for conjugate points in sub-Riemannian manifolds.

In the works mentioned above the symplectic invariants have been computed explicitly only for sub-Riemannian manifolds of \emph{step two}. In the present paper we provide for the first time  explicit expressions of some of these invariants for   sub-Riemannian manifolds of step higher than two.  The importance of these expressions is that they might allow us to test whether or not results which are true in \emph{step two} generalize to \emph{step three} or higher.

Before describing the organization of the paper let us mention that in \cite{BG1} Baudoin and Garofalo introduced a notion of \emph{curvature}, quite different from the one that appears here,  for sub-Riemannian manifolds with \emph{transverse symmetries}. More precisely, they found a  generalization of the \emph{curvature-dimension inequality} from Riemannian geometry which is appropriate for a rich class of sub-Riemannian manifolds.  In \cite{BWa} Baudoin and Wang extended this  notion of curvature  to non-symmetric contact manifolds. In another recent paper, \cite{BKW}, Baudoin, Kim and Wang  established a curvature-dimension inequality on Riemannian foliations with totally geodesic leaves. The interested reader might also consult the following list of references, among many others, for applications of the generalized curvature-dimension inequality: \cite{FBo},\cite{BBG}, \cite{BBGM}, \cite{BG2},  \cite{FK2} and \cite{FK1}.

The organization of the paper is as follows. In sections \ref{Sec:CurvatureOperator} and \ref{Sec:Jacobi} we recall several deep results of the manuscript \cite{ABR1}. For instance, in these sections we recall the construction of the \emph{curvature operator} and the concept of \emph{Jacobi curves}. In Section \ref{Sec:Goursat} we compute explicitly the symplectic invariant $R_{aa,11}$ of the Jacobi curve in  Carnot groups with horizontal Goursat distributions, and therefore, by a crucial result in \cite{ABR1} we  obtain the explicit expression of the invariants $\mathcal{I}_{\lambda}$ and $\mathcal{R}_{\lambda}$. In Section \ref{Engel} we analyze ample and equiregular geodesics in the Engel group. In Section \ref{Sec:Cartan}, we investigate the quantities $\mathcal{I}_{\lambda}$ and $\mathcal{R}_{\lambda}$ in the Cartan group. 

\section{The Curvature Operator}
\label{Sec:CurvatureOperator}

In this section we recall the concept  of  \emph{curvature operator} in sub-Riemannian manifolds. This operator, as we have already mentioned,  was recently introduced by Agrachev, Barilari and Rizzi in the paper \cite{ABR1}. 

\begin{definition}
Let $M$ be a connected, smooth $n-$dimensional manifold. A sub-Riemannian structure on $M$ is a pair $(\mathbb{U},f)$ where:
\begin{enumerate}
\item $\mathbb{U}$ is a smooth rank $k$ Euclidean vector bundle with base $M$ and fiber $\mathbb{U}_x$, i.e. for every $x \in M$, $\mathbb{U}_x$ is a $k-$dimensional vector space endowed with an inner product.
\item $f : \mathbb{U} \rightarrow TM$ is a smooth linear morphism of vector bundles, i.e. $f$ is linear on fibers and the following diagram is commutative:
\[
\xymatrix{
       \mathbb{U}    \ar[dr]_{\pi_{ \mathbb{U}}} \ar[r]^f & TM \ar[d]^{\pi} \\
                               & M }
\]
\end{enumerate}
The maps $\pi_{\mathbb{U}}$ and $\pi$ are the canonical projections of the vector bundles $\mathbb{U}$ and $TM$, respectively. 
\end{definition}

\begin{definition}
\label{dist}
 The distribution $\mathscr{D} \subset TM$ is the family of subspaces
 \[
\mathscr{D}= \{ \mathscr{D}_x \}_{x \in M}, \quad \text{where} \quad \mathscr{D}_x \doteq f(\mathbb{U}_x ) \subset T_xM .
 \]
The family of \emph{horizontal vector fields} $\overline{\mathscr{D}} \subset \mathrm{Vec} (M)$ is
 \[
 \overline{\mathscr{D}}=\mathrm{span} \{f\circ\sigma,\ \sigma:M\rightarrow\mathbb{U}\ \text{is a smooth section of $\mathbb{U}$}\}.
 \]
 \end{definition}

The Euclidean structure on the fibers induces a metric structure on the distribution $\mathscr{D}_x = f(\mathbb{U}_x)$ for all $x\in M$ as  follows:  
\begin{equation}
\label{metric}  
 \|v\|^2_x\doteq \min \Big\{\|u\|^2 \ \big|\ v=f(x,u)\Big\},\quad  \forall v\in\mathscr{D}_x.
\end{equation}    
It is possible to show that $\|\cdot\|$ is a norm on $\mathscr{D}_x$ that satisfies the parallelogram law, i.e., it is actually induced by an inner product $g_x$ on $\mathscr{D}_x$. Notice that the minimum in (\ref{metric}) is always attained since we are minimizing an Euclidean norm in $\R^k$ on an affine subspace.

A Lipschitz continuous curve $\gamma:\left[0,T\right]\rightarrow M$ is horizontal (or \emph{admissible}) if $\dot{\gamma}(t)\in \mathscr{D}_{\gamma(t)}$ for a.e. $t\in \left[0,T\right]$. Given a horizontal curve $\gamma:\left[0,T\right]\rightarrow M$, its \emph{length} is given by
\[
\ell(\gamma)=\int^T_0\left\|\dot{\gamma}(t)\right\|dt.
\]  
Since the length is invariant by reparametrization, we can always assume that $\left\|\dot{\gamma}(t)\right\|$ is constant. The sub-Riemannian \emph{distance} between $x,y\in M$ is the function 
\[
d(x,y)=\left\{\ell(\gamma)\ |\ \gamma(0)=x,\ \gamma(T)=y, \text{ $\gamma$ horizontal}\right\}.
\]
 In this paper we assume that $\mathrm{Lie}_x\overline{\mathscr{D}} = T_xM$, for every $x \in M$. This is the classical bracket-generating (or H\"ormander) condition on the distribution $\mathscr{D}$, which implies the controllability of the system, i.e. $d(x,y) < \infty$ for all $x,y \in M$ (Rashevsky-Chow theorem). Moreover one can show that $d$ induces on $M$ the original manifold's topology. When $(M, d)$ is complete as a metric space, Filippov Theorem guarantees the existence of minimizers joining $x$ to $y$, for all $x, y \in M$.  

Locally, the pair $(\mathscr{D},g)$ can be given by assigning a set of $k$ smooth vector fields $\{X_1,\ldots,X_k\}$ that span $\mathscr{D}$, orthonormal for $g$. In this case, the set $\{X_1,\ldots,X_k\}$ is called a local orthonormal frame for the sub-Riemannian structure. Finally, we can write the system in control form, namely for any horizontal curve $\gamma : [0,T] \rightarrow M$ there is a control $u\in L^{\infty}\left(\left[0,T\right],\R^k\right)$ such that
\[
\dot{\gamma}(t)=\sum^k_{i=1}u_i(t)X_i|_{\gamma(t)},\quad \text{a.e. $t\in \left[0,T\right]$}.
\]
A sub-Riemannian geodesic is an admissible curve $\gamma:\left[0,T\right]\rightarrow M$ such that $\left\|\dot{\gamma}(t)\right\|$ is constant and for any $t\in [0,T]$ there exists an interval  $t\in\left(t_1,t_2\right)\subseteq \left[0,T\right]$, such that the restriction $\gamma|_{\left[t_1,t_2\right]}$ minimizes the length between its endpoints.  Geodesics for which  $\left\|\dot{\gamma}(t)\right\|=1$ are called length parametrized (or of unit speed).

The sub-Riemannian Hamiltonian is written as  
\begin{equation}
\label{sub-Hamiltonian}
H(\ld)\doteq\frac{1}{2}\sum^k_{i=1}\left\langle \ld,X_i\right\rangle^2\quad \forall\ld\in T^*M,
\end{equation}
 where $\left\langle \ld,\cdot\right\rangle$ denotes the action of the covector $\ld$ on vectors. We recall that the definition of $H$ in (\ref{sub-Hamiltonian}) is intrinsic and does not depend on the frame $\{X_1,\ldots,X_k\}$. For $\lambda\in T^{*}M$, $T_{\lambda}\left(T^*M\right)$ is a symplectic vector space with the canonical symplectic form $\sigma_{\lambda}$, defined as the differential of the Liouville form. Recall that the tautological (or Liouville) $1$-form on $T^*M$ is $s\in \Lambda^1(T^*M)$, and it is defined as follows:
 \[
s : \lambda \rightarrow s_{\lambda} \in T_{\lambda}^*(T^*M),\quad \langle s_{\lambda},w\rangle \doteq \langle\lambda,\pi_*w\rangle,\quad  \text{$\lambda \in T^*M, w \in T_{\lambda}T^*M$}, 
 \]
 where $ \pi: T^*M \rightarrow M$ denotes the canonical projection.

Denote by $\vv{a}$ the Hamiltonian vector field on $T^*M$ associated  with a function $a\in C^{\infty}(T^*M)$. The vector field $\vv{a}$ is given by the formula $da=\sigma(\cdot,\vv{a})$. For $i=1,\ldots,k$, let $h_i\in C^{\infty}(T^*M)$ be the linear-on-fiber functions defined by $h_i(\ld)\dot{=}\left\langle \ld,X_i\right\rangle$. Notice that 
\[
 \vv{H}=\sum^k_{i=1}h_i\vv{h}_i.
\]

We recall now a weak version of the Pontryagin Maximum Principle (PMP) in the sub-Riemannian setting (see \cite{PBG}).  This principle tell us that trajectories minimizing the distance between two points are solutions of first-order necessary conditions for optimality. For an elementary proof, the reader can consult the reference \cite{ABB1}. 

\begin{theorem}
\label{PMS}
Let $\gamma:[0,T]\rightarrow M$ be a sub-Riemannian geodesic associated with a non-zero control $u\in L^{\infty}([0,T],\mathbb{R}^k)$. Then there exists a Lipschitz curve $\lambda : [0,T] \rightarrow T^*M$, such that $\pi\circ\ld=\gamma$ and only one of the following conditions holds for a.e. $t \in[0, T ]$:
\begin{enumerate}
\item $\dot{\lambda}(t)=\vv{H}|_{\ld}$, with  $h_i\left(\ld(t)\right)=u_i(t)$,
\item $\dot{\lambda}(t)=\sum^k_{i=1}u_i(t)\vv{h}_i|_{\ld(t)}$,  $h_i\left(\ld(t)\right)=0$.
\end{enumerate}
\end{theorem}
If $\ld : [0,T]\rightarrow M$ is a solution of (1) (resp. (2)) it is called a \emph{normal} (resp. \emph{abnormal}) extremal. It is well known that if $\ld(t)$ is a normal extremal, then its projection $\gamma(t) := \pi(\ld(t))$ is a smooth geodesic. This does not hold in general for abnormal extremals. Notice that extremals satisfying (1) are simply integral lines of the Hamiltonian field $\vv{H}$ . Thus, let $\ld(t) = e^{t\vv{H}} (\ld_0)$ denote the integral line of $\vv{H}$, with initial condition $\ld(0) = \ld_0$. On the other hand, a geodesic can be at the same time normal and abnormal, namely it admits distinct extremals, satisfying (1) and (2). In the Riemannian setting there are no abnormal extremals. 

\begin{definition}
\label{Strongly}
A normal extremal trajectory $\gamma : [0, T ] \rightarrow M$ is called \emph{strictly normal} if it is not abnormal. Moreover, if for all $s \in [0,T]$ the restriction $\gamma|_{[0,s]}$ is also strictly normal, then $\gamma$ is called \emph{strongly normal}.
\end{definition}

\subsection{Geodesic Flag and Young Diagram}

Let $\gamma:\left[0,T\right]\rightarrow M$ be a smooth admissible curve such that $\gamma(0)=x_0$. By definition, this means that $\dot{\gamma}(t)\in\mathscr{D}_{\gamma(t)}$ for all $t\in[0,T]$. Consider a smooth horizontal extension of the tangent vector, namely an horizontal vector field $\mathsf{T} \in \mathscr{D}$ such that $\mathsf{T}|_{\gamma(t)} = \dot{\gamma}(t)$.

For each smooth admissible curve, we consider a family of subspaces, which is related with a micro-local characterization of the sub-Riemannian structure  along the trajectory itself. 
\begin{definition}[\cite{ABR1}]
\label{flag}
The flag of the admissible curve $\gamma$ is the sequence of subspaces 
\[
\mathscr{F}^i_{\gamma}(t)\doteq\mathrm{span}\{\mathcal{L}^j_{\mathsf{T}}(X)|_{\gamma(t)}:X\in\overline{\mathscr{D}},j\leq i-1\}\subset T_{\gamma(t)}M,
\]
where $\mathcal{L}_{\mathsf{T}}$ denotes the Lie derivative in the direction of $\mathsf{T}$.
\end{definition}
Notice that, by definition, this is a filtration of $T_{\gamma(t)}M$, i.e. $\mathscr{F}^i_{\gamma} (t)\subset \mathscr{F}^{ i+1}_{\gamma}(t)$, for all $i\geq1$. Moreover, $\mathscr{F}^1_{\gamma}(t) = \mathscr{D}_{\gamma(t)}$. 
\begin{definition}[\cite{ABR1}]
Let $k_i(t)\doteq\dim\mathscr{F}^i_{\gamma}(t)$. The growth vector of the admissible curve $\gamma$ is the sequence of integers  $\mathcal{G}_{\gamma}(t)=\{k_1(t), k_2(t),\ldots\}$.
\end{definition}

\begin{definition}[\cite{ABR1}]
\label{ampledef}
Let $\gamma:[0,T]\rightarrow M$ be an admissible curve, with growth vector $\mathcal{G}_{\gamma}(t)$. We say that the geodesic is:
\begin{enumerate}
\item \emph{Equiregular} at $t$ if its growth vector is locally constant at $t$. 
\item \emph{Ample} at $t$  if there exists an integer $m = m(t)$ such that $\mathscr{F}^{m(t)}_{\gamma}(t)=T_{\gamma(t)}M$. We call the minimal $m(t)$ such that the curve is ample \emph{the step at $t$ of the admissible curve}.
\end{enumerate}
Finally, an admissible curve is ample (resp. equiregular ) if it is ample (resp. equiregular) at each $t \in [0,T]$.
\end{definition}

We stress that equiregular (resp. ample) geodesics are the microlocal counterpart of equiregular (resp. bracket-generating) distributions.

Let 
\[
d_i\doteq\dim\mathscr{F}^i_{\gamma}(t)-\dim\mathscr{F}^{i-1}_{\gamma}(t), 
\]
for $i\geq 1$, be the
increment of dimension of the flag of the geodesic at each step (with the convention $k_0\doteq 0$). For an  equiregular geodesic we have $d_1\geq d_2\geq \ldots\geq d_m$. This result was proved  in \cite[Appendix E]{ABR1}. Moreover, if the geodesic is also ample we have $\sum^{m}_{i=1}d_i=n$. 

For any  ample and  equiregular geodesic $\gamma:[0,T]\rightarrow M$, we draw a \emph{Young diagram} $D$ with $d_i$ blocks in the $i$-th column, with $i\geq 1$, and we define $n_1,\ldots,n_k$ as the lengths of its rows (that may depend on $\gamma$).

\subsection{Curvature operator} In order to carry out the construction of the \emph{curvature operator} done in \cite{ABR1}, we recall the definitions of \emph{geodesic cost} and  of \emph{second differential} of a function.

\begin{definition}
The geodesic cost associated with a strongly normal geodesic $\gamma:[0,T]\rightarrow M$ with $\gamma(0)=x_0$ is the family of functions
\[
c_t(x)\doteq-\frac{1}{2t}d^2(x,\gamma(t)),\quad x\in M,t>0.
\]
\end{definition}

The following theorem, whose proof can be found in \cite{ABR2}, gives us the regularity of the geodesic cost, and more importantly, it will allow us to define its Hessian.

\begin{theorem} 
\label{critical-ct}
 Let $x_0 \in M$ and $\gamma:[0,T]\rightarrow M$  be  a strongly normal geodesic with $\gamma(0)=x_0$ and initial covector $\lambda_0$. Then, there exists $\varepsilon> 0$ and an
open set $U \subset (0, \varepsilon) \times M$ such that
 \begin{enumerate}
\item $(t,x_0) \in U$ for all $t \in (0,\varepsilon)$,
\item the function $(t,x)\rightarrow c_t(x)$ is smooth on $U$,
\item For any $(t,x) \in U$, the covector $\lambda_x = d_xc_t$ is the initial covector of the unique minimizing geodesic connecting $x$ with $\gamma(t)$ in time $t$.
 \end{enumerate}
In particular, $d_{x_0}c_t=\lambda_0$ and $x_0$ is a critical point for the function $\dot{c}_t \doteq \frac{d}{dt} c_t$ for every $t \in (0, \varepsilon)$.
\end{theorem}

Now, consider a smooth function  $f : M\rightarrow \R$. Its first differential at a point $x \in M$ is the linear map $d_x f : T_x M\rightarrow \R$. The second differential of $f$ is well defined only at a critical point, i.e., at those points $x$ such that $d_xf = 0$. Indeed, in this case the map
\[
d^2_x f : T_xM \times T_x M\rightarrow \R,\quad d^2_x f(v,w)=V(Wf)( x ),
\]
where $V, W$ are vector fields such that $V(x) = v$ and $W (x) = w$, respectively, is a well defined symmetric bilinear form that does not depend on the choice of the extensions. The associated quadratic form, that we denote by the same symbol $d^2_xf : T_xM \rightarrow \R$, is defined by
\[
d^2_xf(v)=\frac{d^2}{dt^2}\Big|_{t=0}f(\alpha(t)),\quad \text{where $\alpha(0)=x$ and $\dot{\alpha}(0)=v$}.
\]
By Theorem \ref{critical-ct}, for every $t\in (0,\varepsilon)$, the function $x\rightarrow \dot{c}_t(x)$ has a critical point at $x_0$. Remember that $\lambda_0$ is the initial covector of the geodesic $\gamma$. Now, using the inner product $g$ on $\mathscr{D}_x$ we can associated with $d^2_{x_0}\dot{c}_t(v)|_{\mathscr{D}}$ a family of symmetric operators on the distribution $\mathcal{Q}_{\lambda_0}(t)$ defined by the identity
\[
d^2_{x_0}\dot{c}_t(v)\doteq g\left( \mathcal{Q}_{\lambda_0}(t)v,v\right)_{\lambda_0}, \quad t\in(0,\varepsilon),v\in\mathscr{D}_{x_0}.
\]
The following theorems, namely Theorem \ref{TheoremA}, Theorem \ref{TheoremB} and Theorem \ref{TheoremC}, are among the main results in the already mentioned paper \cite{ABR1}. These important results are valid for a wide class of optimal control problems.
\begin{theorem}
\label{TheoremA}
Let $\gamma : [0,T] \rightarrow M$ be an ample geodesic at $t=0$, with initial covector $\lambda_0\in T^* M$. The map $t\rightarrow t^2\mathcal{Q}_{\lambda_0(t)}$ can be extended to a smooth family of operators on $\mathscr{D}_{x_0}$ for small $t \geq 0$, symmetric with respect to $g$. Moreover,
\[
\mathcal{I}_{\lambda_0}\dot{=}\lim_{t\rightarrow 0^+}t^2\mathcal{Q}_{\lambda_0(t)}\geq \mathbb{I}>0
\]
as operators on $\left(\mathscr{D}_{x_0},g_{x_0}\right)$. Finally,
\[
\frac{d}{dt}\Big|_{t=0}t^2\mathcal{Q}_{\lambda_0(t)}=0.
\]
\end{theorem}
\begin{definition}
The curvature operator $\mathcal{R}_{\lambda_0}:\mathscr{D}_{x_0}\rightarrow\mathscr{D}_{x_0}$ at $\lambda_0\in T^*_{x_0}M$ is defined by
\begin{equation*}
\label{curvature}
\mathcal{R}_{\lambda_0}\doteq\frac{3}{2}\frac{d^2}{dt^2}\Big|_{t=0}t^2\mathcal{Q}_{\lambda_0}(t).
\end{equation*}
Moreover, the Ricci curvature at $\lambda_0\in T^*_{x_0}M$ is defined by $Ric\left(\lambda_0\right)\dot{=}\mathrm{tr}\ \mathcal{R}_{\lambda_0}$.
\end{definition}

In particular, we have the following Laurent expansion for the family of symmetric operators $\mathcal{Q}_{\lambda_0}(t):\mathscr{D}_{x_0}\rightarrow\mathscr{D}_{x_0}$
\begin{equation*}
\mathcal{Q}_{\lambda_0}(t)=\frac{1}{t^2}\mathcal{I}_{\lambda_0}+\frac{1}{3}\mathcal{R}_{\lambda_0}+O(t) \quad t>0.
\end{equation*} 

\begin{theorem}
\label{TheoremB}
Let $\gamma:[0,T]\rightarrow M$ be an ample and equiregular geodesic. Then, the symmetric operator $\mathcal{I}_{\lambda_0}:\mathscr{D}_{x_0}\rightarrow\mathscr{D}_{x_0}$ satisfies
\begin{enumerate}
\item \emph{spec} $\mathcal{I}_{\lambda_0}=\{n^2_1,\ldots,n^2_k\}$,
\item \emph{tr} $\mathcal{I}_{\lambda_0}=\sum^k_{i=1}n^2_i$,
\end{enumerate}
where $n_1,\ldots, n_k$ are the lengths of the rows of  the associated Young diagram.
\end{theorem}

Let $\Delta_{\mu}$ be the sub-Laplacian associated with a smooth volume form $\mu$. The next result is an explicit expression for the asymptotic of the sub-Laplacian of the squared
distance from a geodesic, computed at the initial point $x_0$ of the geodesic $\gamma$. Let $\mathfrak{f}_t(\cdot)\dot{=}\frac{1}{2}d^2\left(\cdot,\gamma(t)\right)$. 

\begin{theorem}
\label{TheoremC}
Let $\gamma$ be an ample and equiregular geodesic with initial covector $\lambda_0\in T^*_{x_0}M$.   Assume 
also that $\dim \mathscr{D}$ is constant in a neighborhood of $x_0$. Then there exists a smooth n-form $\omega$ defined along $\gamma$, such that for any volume form $\mu$ on M, $\mu_{\gamma(t)} = e^{g(t)}\omega_{\gamma(t)}$,  we have
\[
\Delta_{\mu}\mathfrak{f}_t|_{x_0}=\operatorname{tr}\mathcal{I}_{\gamma}-\dot{g}(0)t-\frac{1}{3}\operatorname{Ric}(\lambda_0)t^2+O(t^2).
\]
\end{theorem}
In the following section we recall how to  find the explicit expression of the \emph{curvature operator} $\mathcal{R}_{\lambda_0}$ when the geodesic $\gamma$ is also equiregular.  In \cite {ABR1} it was shown that in order to obtain such expression of the operator we need to compute certain symplectic invariants associated with the \emph{Jacobi curve}. 

\section{The Curvature Operator: Jacobi Curves}
\label{Sec:Jacobi}

In this section we recall the concept of \emph{Jacobi curve} associated with a normal geodesic, which is a curve of Lagrangian subspaces in a symplectic vector space. We also introduce a key technical tool, the so-called \emph{canonical frame}, associated with a monotone, ample, equiregular Jacobi curve. 

Let $(\Sigma,\sigma)$ be a $2n-$dimensional symplectic vector space. A subspace $\Lambda\subset\Sigma$ is called Lagrangian if it has dimension $n$ and $\sigma|_{\Lambda} \equiv 0$. The Lagrange Grassmannian $L(\Sigma)$ is the set of all $n-$dimensional Lagrangian subspaces of $\Sigma$.

Fix now $\Lambda \in L(\Sigma)$. The tangent space $T_{\Lambda}L(\Sigma)$ to the Lagrange Grassmannian at the point $\Lambda$ can be canonically identified with the set of quadratic forms on the space $\Lambda$ itself, namely
\[
T_{\Lambda}L(\Sigma) \simeq Q(\Lambda).
\]
Indeed, consider a smooth curve $\Lambda(\cdot)$ in $L(\Sigma)$ such that $\Lambda(0) = \Lambda$, and denote by $\dot{\Lambda} \in T_{\Lambda}L(\Sigma)$ its tangent
vector. For any point $z \in \Lambda$ and any smooth extension $z(t) \in \Lambda(t)$, we define the quadratic form 
\[
\dot{\Lambda} \doteq z   \rightarrow \sigma ( z , \dot{z} ) ,
\]
where $\dot{z} \doteq \dot{z}(0)$. A simple check shows that the definition does not depend on the extension $z(t)$.

Let $J(\cdot) \in L(\Sigma)$ be a smooth curve in the Lagrange Grassmannian. For $i \in \mathbb{N}$, consider
\[
J^{(i)}(t)=\mathrm{span}\ \Big\{  \frac{d^j}{d t^j}\ell(t)\Big|\  \ell(t)\in J(t),\ \ell(t)\ \mathrm{smooth},\ 0\leq j\leq i \Big\}\subset\Sigma,\quad  i\geq0.
\]
\begin{definition} 
The subspace $J^{(i)}(t)$ is the $i-$th extension of the curve $J(\cdot)$ at $t$. The flag
\[
 J(t) = J^{(0)}(t) \subset J^{(1)}(t) \subset J^{(2)}(t) \subset \ldots \subset \Sigma,
 \]
 is the associated flag of the curve at the point $t$. The curve $J(\cdot)$ is called:
 \begin{enumerate}
\item equiregular at $t$ if $\mathrm{dim}\ J^{(i)}(\cdot)$ is locally constant at $t$, for all $i \in\mathbb{N}$, 
\item ample at $t$ if there exists $N \in\mathbb{N}$ such that $J^{(N)}(t) = \Sigma$,
\item monotone increasing (resp. decreasing) at $t$ if $\dot{J}(t)$ is non-negative (resp. non-positive) as a quadratic form.
\end{enumerate}
The step of the curve at $t$ is the minimal $N \in\mathbb{N}$ such that $J^{(N)}(t) = \Sigma$.
\end{definition}

The Jacobi curve arises from the geometric interpretation of the second derivative of the \emph{geodesic cost}.  Thus, in order to define the Jacobi curve we need to recall the notion of \emph{second differential} of a function $f\in C^{\infty}(M)$ at non-critical points.

\begin{definition}
 Let $f \in C^{\infty}(M)$, and
 \[
df :M\rightarrow T^*M, \quad df :x\rightarrow d_x f.
 \]
Fix $x \in M$, and let $\lambda \doteq d_xf \in T^*M$. The second differential of $f$ at $x\in M$ is the linear map 
\[
d^2_xf  \doteq d_x(df) : T_xM\rightarrow T_{\lambda}(T^*M),\quad  d^2_x f : v\rightarrow \frac{d}{ds}\Big|_{s=0}d_{\gamma(s)}f,
\]
where $\gamma(\cdot)$ is a curve on $M$ such that $\gamma(0)=x$, and $\dot{\gamma}(0)=v$.
\end{definition}

Let $\gamma:[0,T]\rightarrow M$ be a strongly normal geodesic, with $\gamma(0)=x_0\in M$. Without loss of generality, we can choose $T$ sufficiently small so that the map $(t,x)\rightarrow c_t(x)$ is smooth in a neighborhood of $(0,T)\times\{x_0\}\subset\mathbb{R}\times M$, and $d_xc_t=\lambda_0$ is the initial convector associated with $\gamma$, i.e.,
\[
\lambda(t)=e^{t\vv{H}}(\lambda_0), \quad \gamma(t)=\pi(\lambda(t)), \quad \left(\text{i.e} \quad \lambda(t)\in T^{*}_{\gamma(t)}(M)\right).
\]
For any $\ld\in T^*M$, $\pi(\lambda)=x$, we denote with the symbol  $\mathcal{V}_{\lambda}\equiv T_{\lambda}(T^{*}_x M)$  the vertical subspace at the point $\lambda\in T^{*}M$, i.e. the tangent space at $\lambda$ to the fiber $T^*_x M$. Observe that, if $\pi:T^*M\rightarrow M$ is the bundle projection, $\mathcal{V}_{\lambda}=$ ker $\pi_{\ast}|_{T_{\lambda}(T^{*} M)}$.

\begin{definition}
The Jacobi curve associated with $\gamma$ is the smooth curve $J_{\lambda_0} : [0,T] \rightarrow L(T_{\lambda_0}(T^*M))$ defined by
\begin{equation}
J_{\lambda_0}(t)\doteq d^2_{x_0}c_t\left(T_{x_0}M\right),
\end{equation} 
for $t\in (0,T],$ and $J_{\lambda_0}(0)\doteq\mathcal{V}_{\lambda_0}.$
\end{definition} 
 Now, let $v\in T_{x_0}M$ and $\alpha$ a smooth curve such that $\alpha(0)=x_0$ and $\dot{\alpha}(0)=v$.  For $s$ small enough, $d_{\alpha(s)}c_t$ is the initial covector of the unique normal geodesic that connects $\alpha(s)$ and $\gamma(t)$ in time $t$, or in other words, $\pi\circ e^{t\vv{H}}\circ d_{\alpha(s)}c_t=\gamma(t)$. Then, 
\[
\pi_*\circ e^{t\vv{H}}_*\circ d^2_x c_t(v)=\frac{d}{ds}\Big|_{s=0}\pi\circ e^{t\vv{H}}\circ d_{\alpha(s)}c_t=0.
\]
Therefore, one can actually write 
\begin{equation}
\label{JVt}
J_{\lambda_0}(t)=e^{-t\vv{H}}_{\ast}\mathcal{V}_{\ld(t)}.
\end{equation} 
Moreover, from Proposition 6.12 in \cite{ABR1}, we have that the Jacobi curve $J_{\lambda}$ is monotone decreasing for every $\lambda\in T^{\ast}M$.

The following proposition provides the connection between the flag of a normal geodesic and the flag of the associated Jacobi curve, see Proposition 6.15 in \cite{ABR1}.
\begin{proposition}
 Let $\gamma(t) = (\pi\circ\lambda)(t)$ be a normal geodesic associated with the initial covector $\lambda_0$. The flag of the Jacobi curve $J_{\lambda_0}$ projects to the flag of the geodesic $\gamma$ at $t = 0$, namely
 \[
 \pi_{\ast}J^{(i)}_{\lambda_0}(0) = \mathscr{F}^i_{\gamma}(0), \quad \forall i \in \mathbb{N}. 
 \]
Moreover, $\dim J^{(i)}_{\lambda_0}(t) = n + \dim \mathscr{F}^i_{\gamma}(t)$. Therefore $\gamma$ is ample of step $m$ (resp. equiregular) if and only if $J_{\lambda_0}$  is ample of step $m$ (resp. equiregular).
\end{proposition}

In order to compute explicitly the expansion of the operator $\mathcal{Q}(t)$ we need to recall the concept of \emph{canonical frame} along the curve $J_{\lambda_0}$. This frame generalizes the concept of parallel transport from Riemannian geometry to (sufficiently regular) sub-Riemannian extremals. The canonical frame was first introduced by Agrachev and Gamkrelidze in \cite{AG}, Agrachev and Zelenko in \cite{AZ} and successively extended by Zelenko and Li in \cite{ZL}.

Consider an ample, equiregular geodesic, with Young diagram $D$, with $k$ rows, of length $n_1,\ldots,n_k$. Indeed $n_1 + \ldots + n_k = n$. The moving frame we are going to introduce is indexed by the boxes of the Young diagram, so we need to fix some terminology first. Each box is labelled ``$ai$'', where $a = 1,\ldots,k$ is the row index, and $i = 1,\ldots,n_a$ is the progressive box number, starting from the left, in the specified row. Briefly, the notation $ai \in D$ denotes the generic box of the diagram. We employ letters from the beginning of the alphabet $a, b, c, \ldots$ for rows, and letters from the middle of the alphabet $i, j, h,\ldots$ for the position of the box in the row.

We collect the rows with the same length in $D$, and we call them levels of the Young diagram. In particular, a level is the union of $r$ rows $D_1,\ldots,D_r$, and $r$ is called the size of the level. The set of all the boxes $ai \in D$ that belong to the same column and the same level of $D$ is called superbox. We use greek letters $\alpha, \beta,\ldots$ to denote superboxes. Notice that that two boxes $ai, bj$ are in the same superbox if and only if $ai$ and $bj$ are in the same column of $D$ and in possibly distinct row but with same length, i.e. if and only if $i = j$ and $n_a = n_b$.

\begin{theorem}
\label{Can-Frame}
There exists a smooth normal moving  frame $\left\{E_{ai}(t), F_{ai}(t)\right\}_{ai\in D}$ of a monotonically nonincreasing ample and equiregular curve $J(\cdot)$ with given Young diagram $D$, with $k$ rows, of length $n_a$, for $a=1,\ldots,k$, such that
\begin{enumerate}
\item
 $J(t) = \text{span}\left\{E_{ai}(t)\right\}_{ai\in D}$ for
any $t$.
\item
 It is a Darboux basis, namely, 
 \[
 \sigma(E_{ai},E_{bj})=\sigma(F_{ai},F_{bj})=\sigma(E_{ai},F_{bj})-\delta_{ab}\delta_{ij}=0,
 \]
 for every $ai,bj\in D$.
\item
 There exists a one-parametric family of $n\times n$, with $n_1+\ldots+n_k=n$, symmetric matrices $R(t)$, with components $R_{ab,ij}(t) = R_{ba,ji}(t)$, indexed by the boxes of the Young diagram D such that the moving frame satisfies the structural equations
\begin{align*}
\dot{E}_{ai}&= E_{a(i-1)}(t),  &a&=1,...,k,i=2,...,n_a,\\
\dot{E}_{a1}&=- F_{a1}(t), & a&=1,...,k,\\
\dot{F}_{ai}&=\sum^{k}_{b=1}\sum^{n_b}_{j=1}R_{ab,ij}(t)E_{bj}(t)-F_{a(i+1)}(t), & a&=1,...,k,i=1,...,n_a-1,\\
\dot{F}_{an_a}&=\sum^{k}_{b=1}\sum^{n_b}_{j=1}R_{ab,n_aj}(t)E_{bj}(t), & a&=1,...,k.
\end{align*}
The matrix R(t) is normal in the sense of \cite{ZL}. See also \cite[Appendix F]{ABR1}, where the normal condition
for $R(t)$ is written explicitly in this notation.
\end{enumerate}
Properties (1)-(3) uniquely define the frame up to orthogonal transformation that preserve the
Young diagram. More precisely, if $\{\widetilde{E}_{ai},\widetilde{F}_{ai}\}_{ai\in D}$ is another smooth moving frame along $\lambda(t)$ satisfying (1)-(3), for some family $\widetilde{R}(t)$, then for any superbox $\alpha$ of size $r$ there exists an orthogonal (constant) $r \times r$ matrix $O^{\alpha}$ such that 
\[
\widetilde{E}_{ai}=\sum_{bj\in\alpha}O^{\alpha}_{ai,bj}E_{bj},\quad \widetilde{F}_{ai}=\sum_{bj\in\alpha}O^{\alpha}_{ai,bj}F_{bj}, \quad ai\in\alpha.
\]
\end{theorem}

Any canonical Darboux frame $\{E_{ai},F_{ai}\}_{ai\in D}$  defines a Lagrangian splitting  $\Sigma_{\lambda_0}=\mathcal{V}_{\lambda_0}\oplus\mathcal{H}_{\lambda_0}$, where
\[
\mathcal{V}_{\lambda_0}=\text{span}\{E_{ai}(0)\}_{ai\in D}, \quad \mathcal{H}_{\lambda_0}=\text{span}\{F_{ai}(0)\}_{ai\in D}.
\]
Observe that $\mathcal{V}_{\lambda_0} = J_{\lambda_0}(0) = \mathrm{ker}\ \pi_*|_{T_{\lambda}(T^*M)}$, and $\pi_*$ induces an isomorphism between $\mathcal{H}_{\lambda_0}$ and $T_{x_0}M$.  Now, the curve $J_{\lambda_0}(t)$ is the graph of a linear map $S(t) : \mathcal{V}_{\lambda_0} \rightarrow \mathcal{H}_{\lambda_0}$ for small $t\geq 0$. Equivalently, by  \cite[Lemma 6.3]{ABR1}, for $0 < t < \varepsilon$, $J_{\lambda_0}(t)$ is the graph of $S(t)^{-1} : \mathcal{H}_{\lambda_0}\rightarrow \mathcal{V}_{\lambda_0}$. We stress that the function $S(t)^{-1}$ is defined only for $t>0$ sufficiently small.

From the definition of second differential we have that if $\alpha(\cdot)$ is a smooth arc with $\alpha(0)=x_0$ and $\dot{\alpha}(0)=v\in T_{x_0}M$, then
\[
\pi_*d^2_{x}c_t(v)=\frac{d}{ds}\Big|_{s=0}\pi\circ d_{\alpha(s)}c_t=\frac{d}{ds}\Big|_{s=0}\alpha(s)=v.
\]
Fix $v\in T_{x_0}M$ and let $\widetilde{v}\in\mathcal{H}_{\lambda_0}$ be the unique horizontal lift such that $\pi_{\ast}\widetilde{v}=v$. Hence, for $t>0$ we have
\[
d^2_{x_0}c_t(v)=S(t)^{-1}\widetilde{v}+\widetilde{v}.
\]
Then, by  the standard identification $\mathcal{V}_{\lambda_0}\simeq T^*_{x_0}M$ and the fact that, for $\xi\in \mathcal{V}_{\lambda}$ and for any $X\in T_{\lambda}(T^*M)$ then $\sigma(\xi,X)=\langle \xi,\pi_*X\rangle$, we finally get 
\begin{equation}
\label{Q=S}
g\left(\mathcal{Q}_{\lambda_0}(t)v,v\right)=\frac{d}{dt}\sigma\left(S(t)^{-1}\widetilde{v},\widetilde{v}\right),\quad \text{for  $v\in\mathscr{D}_{x_0}$  and $t>0$}.
\end{equation}
Since $J_{\lambda_0}(0)=\mathcal{V}_{\lambda_0}$, it follows that $S^{-1}(t)$ is singular at $t=0$. Now, let $X_a = \pi_*F_{a1}(0) \in T_xM$. Then, from \cite[Lemma 7.9]{ABR1} the set $\{X_a\}^k_{a=1}$ is an orthonormal basis for $(\mathscr{D}_x , g)$. Hence, if $v=\sum^k_{a=1}v_aX_a\in\mathscr{D}_x$, we have $\widetilde{v}=\sum^k_{a=1}v_aF_{a1}(0)$. In Young diagram notation, we have $S(t)=S(t)_{ab,ij}$. Thus, we get from (\ref{Q=S}) that
\[
g\left(\mathcal{Q}_{\lambda_0}(t)v,v\right)=\frac{d}{dt}\sum^k_{a,b=1}S(t)^{-1}_{ab,11}v_av_b,\quad t>0.
\]
For convenience,  introduce for $t>0$ the smooth family of $k\times k$ matrices  $S^{\flat}(t)^{-1}$ defined by
\[
S^{\flat}(t)^{-1}\doteq [S(t)^{-1}]_{ab,11}, \quad t>0.
\]
Then, the quadratic form associated with the operator $\mathcal{Q}_{\lambda_0}:\mathscr{D}_{x_0}\rightarrow\mathscr{D}_{x_0}$ via the Hamiltonian inner product is represented by the matrix $\frac{d}{dt}S^{\flat}(t)^{-1}$. The following crucial result, see   \cite[Corollary 7.5]{ABR1}, connects the  \emph{curvature operator} with the invariants of the \emph{Jacobi curve}, since it gives the asymptotic expansion of $S^{\flat}(t)^{-1}$ in terms of the symplectic invariants $R(t)$ of the canonical frame.
\begin{theorem}
\label{Sbeta}
Let $\gamma(\cdot)$ be an ample and equiregular geodesic with a given Young diagram $D$ with $k$ rows, of length $n_a$, for $a = 1,...,k.$ Then, for $0<|t|<\varepsilon$ 
\begin{equation}
\label{fSbeta1}
S^{\flat}(t)^{-1}_{ab}=-\delta_{ab}\frac{n^2_a}{t}+R_{ab,11}(0)\Omega(n_a,n_b)t+O(t^2),
\end{equation}
where
\begin{equation}
\Omega(n_a,n_b)=\left\{ 
  \begin{array}{l l}
    0, & \quad |n_a-n_b|\geq 2\\
    \frac{1}{4(n_a+n_b)}, & \quad |n_a-n_b|=1  \\
    \frac{n_a}{4n^2_a-1},& \quad n_a=n_b.
  \end{array} \right.
\label{fSbeta2}
\end{equation}
\end{theorem}

\section{Carnot groups with horizontal distribution of Goursat type}
\label{Sec:Goursat}

 In the following, we will compute the part of the canonical curvature $R(t)$ that is relevant for computing the curvature operator for Carnot groups $\mathbb{J}^{n}, n\geq 3$, with horizontal distribution of \emph{Goursat-type}. These groups are $n$-dimensional Carnot groups of $(n-1)$-step  with two dimensional horizontal sub-bundle. 
 
 In $\R^n$ with coordinates $\left(x,y_0,y_1,\ldots,y_{n-2}\right)$,  we consider the vector fields
 \[
 X_1=\frac{\pa}{\pa x}, \quad X_{i+2}=\sum^{n-2}_{j=i}\frac{x^{j-i}}{(j-i)!}\frac{\pa}{\pa y_j}, \quad i=0,\ldots,n-2. 
 \]
 The vector fields $X_1$ and $X_2$ are called \emph{horizontal} and satisfy the H\"ormander condition, i.e. they generate the whole tangent bundle by their commutators: 
\begin{equation}
\label{Comm-G}
\left[X_1,X_{i}\right]=X_{i+1}, \quad \text{for $i=2,\ldots,n-1$},
\end{equation} 
and all other commutators being  zero.

We can introduce a unique Lie group structure (law of multiplication) in $\R^n$, making a Lie group of $\R^n$,  so that $X_1, \ldots, X_n$ become basic left-invariant  fields on this Lie group. We denote this Lie group by $\mathbb{J}^n$.  We have that $\mathscr{D}|_h\doteq\mathrm{span}\ \{X_1,X_2\}|_h$, for all $h\in \mathbb{J}^n$, is a distribution left-invariant by the action $L_g:\mathbb{J}^n\rightarrow\mathbb{J}^n$. Any left-invariant  scalar product $g$ on $\mathscr{D}|_h$ induces a left-invariant sub-Riemannian structure $(\mathscr{D},g)$ on $\mathbb{J}^n$. Since any two different choices give rise to isometric sub-Riemannian structures,  we choose without loss of generality $g$ such that $X_1$ and $X_2$ are orthonormal. 
 
 Now let $\{\nu_1,\ldots,\nu_n\}$ be the dual frame  of $\{X_1,\ldots, X_n\}$. This dual frame induces coordinates $\left\{h_1,...,h_{n}\right\}$ in each fiber of $T^*\mathbb{R}^n$,
\[
\ld=\left( h_1,\ldots,h_n\right)\quad\Longleftrightarrow\quad \ld=h_1\nu_1+\ldots+h_n\nu_n 
\]
where $h_i(\ld)=\left\langle \ld,X_i\right\rangle$ are the linear-on-fibers functions associated with $X_i$, for $i=1,\ldots,n$. 

Let $\vv{h}_i\in \text{Vec}(T^*\mathbb{R}^n)$ be the Hamiltonian vector fields associated with $h_i\in C^{\infty}(T^*\mathbb{R}^n)$ for $i=1,\ldots,n$, respectively. Consider the vertical vector fields $\pa_{h_i}\in \text{Vec}(T^*\mathbb{R}^n)$, for $i=1,\ldots,n$. The vector fields 
\[
\vv{h}_1,\ldots,\vv{h}_n,\pa_{h_1},\ldots,\pa_{h_n},
\]
are a local frame of vector fields of $T^*\mathbb{R}^n$. Equivalently, we can introduce the cylindrical coordinates $\rho,\theta,h_3,\ldots,h_n$ on each fiber of $T^*\mathbb{R}^n\setminus\{0\}$ by 
\[
h_1=\rho\cos\left(\theta+\frac{\pi}{2}\right),\quad h_2=\rho\sin\left(\theta+\frac{\pi}{2}\right),
\]
with $\rho\in (0,+\infty)$ and $\theta \in (-\pi,\pi]$, and employ instead the local frame
\[
\vv{h}_1,\ldots,\vv{h}_n,\pa_{\theta}, \pa_{\rho},\pa_{h_3},\ldots,\pa_{h_n},
\]
where $\pa_{\theta}\doteq h_1\pa_{h_2}-h_2\pa_{h_1}$. Finally, let the \emph{Euler vector field} be given by
\[
\mathfrak{e}\doteq \sum^n_{i=1}h_i\pa_{h_i}=\rho\pa_{\rho}+\sum^{n}_{i=3}h_i\pa_{h_i}.
\]
Notice that $\mathfrak{e}$ is a vertical vector field on $T^*\mathbb{R}^n$, i.e. $\pi_*\mathfrak{e}=0$, and is the generator of the dilations  $\ld\mapsto e^c\ld$ along the fibers of $T^*\mathbb{R}^n$. 

Notice that  the symplectic form $\sigma$ in the cylindric coordinates $\rho,\theta,h_3,\ldots,h_n$ has the following expression:
\begin{eqnarray*}
\sigma&=&\rho d\rho\wedge\nu_{\bar{\theta}}-\rho^2d\theta\wedge\nu_{\theta}+\sum^{n}_{i=3}dh_i\wedge\nu_i+\rho^2 h_3\nu_{\bar{\theta}}\wedge\nu_{\theta}+\\
&&+\sum^{n}_{i=4}h_{i}\left(h_2\nu_{i-1}\wedge\nu_{\theta}+h_1\nu_{i-1}\wedge\nu_{\bar{\theta}}\right).
\end{eqnarray*}
where $\{\nu_{\theta},\nu_{\bar{\theta}}\}$ is the dual co-frame associated with the frame $\{X_{\bar{\theta}},X_{\theta}\}$, where $X_{\bar{\theta}}\doteq h_1X_1+h_2X_2$ and $X_{\theta}\doteq h_2X_1-h_1X_2$. If we use the formula $dH=\sigma\left(\cdot,\vv{H}\right)$, we can write the explicit expression for the Hamiltonian vector field $\vv{H}$ as follows:
\begin{equation}
\label{ExpH}
\vv{H}=X_{\bar{\theta}}+h_3\pa_{\theta}+\sum^{n-1}_{i=3}h_1h_{i+1}\pa_{h_i}.
\end{equation}

The following proposition will be useful for the characterization of ample/equiregular geodesics in $\J^n$ and the Cartan group $\mathfrak{C}$. For a proof see   \cite[Proposition 3.12]{ABR1} and  \cite[Section 1.3]{Cor}.
\begin{proposition}
\label{RAg}
For any smooth geodesic  $\gamma:[0,T] \rightarrow M$, on a real-analytic structure, such as Carnot groups, we have that the following properties are equivalent: ample at 0, ample at $t \in[0,T]$, strictly normal, strongly normal, not abnormal.
\end{proposition}

 In the following theorem we study when an ample geodesic is equiregular at a time $t$.

\begin{theorem}
\label{NonVanh}
Let  $\gamma=\pi\circ\lambda:[0,T]\rightarrow \mathbb{R}^n$, $n\geq 4$, be a normal   geodesic. Then, 
\begin{enumerate}
\item If there is no time $t$ such that $h_1(t)$ and $h_3(t)$ are both zero, then $\gamma$ is equiregular at $t\in [0,T]$ if and only if $h_1(t)\neq0$. Moreover, $\gamma$ is ample and the geodesic growth vector is given by
\begin{equation}
\label{GrGoursat}
\mathcal{G}_{\gamma}(t)=\left\{ 
  \begin{array}{l l}
    (2,3,4,\ldots,j,j+1,\ldots,n) & \quad \text{if $h_1(t)\neq 0$},\\
   (2,3,3,4,4,\ldots,j,j,\ldots,n) & \quad \text{if $h_1(t)= 0$}. 
  \end{array} \right.
\end{equation}
\item If there is a point $\bar{t}\in[0,T]$ such that $h_1(\bar{t})=h_3(\bar{t})=0$, then $h_1\equiv h_3\equiv0$. In this case, the geodesic is not ample, and we are in the case of an abnormal geodesic. 
\end{enumerate}
\end{theorem}
\begin{proof}
Recall that the geodesic equations in the group $\J^n$ are:
\begin{equation}
\label{vereqg}
\dot{h}_1=-h_2h_3, \quad \dot{h}_i=h_1h_{i+1}, \quad i=2,\ldots,n-1,\quad \dot{h}_n=0.
\end{equation}

\begin{enumerate}

\item Let us now consider the case when there is no time $t$ such that $h_1(t)$ and $h_3(t)$ are both zero. Recall that an admissible extension of $\dot{\gamma}$, where 
\[
\dot{\gamma}(t)=h_1(t)X_1|_{\gamma(t)}+h_2(t)X_2|_{\gamma(t)},
\]
 is a smooth vector field $\mathsf{T}$, of the form $\mathsf{T}=v_1X_1+v_2X_2$, where $v_i\in C^{\infty}(\R^n)$ with $v_i(\gamma(t))=h_i(t)$ for $i=1,2$. Let $X\in\mathscr{D}$ and suppose that it is given by $X=w_1X_1+w_2X_2$, where $w_i\in C^{\infty}(\R^n)$ for $i=1,2$. It is easy to see that for $j\geq 1$,
\begin{equation*}
\label{LjX}
\mathcal{L}^j_{\mathsf{T}}(X)|_{\gamma(t)}=w_1(\gamma(t))\mathcal{L}^j_{\mathsf{T}}(X_1)|_{\gamma(t)}+w_2(\gamma(t))\mathcal{L}^j_{\mathsf{T}}(X_2)|_{\gamma(t)}\mod \mathscr{F}^{j}_{\gamma}(t).
\end{equation*}

A simple computation gives
\begin{eqnarray}
\mathcal{L}_{\mathsf{T}}(X_1)|_{\gamma(t)}&=&-h_2X_3\mod\mathscr{D}_{\gamma(t)}, \label{l1X1}\\
\mathcal{L}_{\mathsf{T}}(X_2)|_{\gamma(t)}&=&h_1X_3\mod\mathscr{D}_{\gamma(t)}. \label{l1X2}
\end{eqnarray}

\begin{enumerate}
\item Assume that $h_1(t)\neq 0$.  By using the geodesic equations (\ref{vereqg}) together with (\ref{l1X1}) and (\ref{l1X2}), we can easily prove by induction  that
\begin{equation*}
\mathcal{L}^{j}_{\mathsf{T}}(X_1)|_{\gamma(t)}=-h^{j-1}_1(t)h_2(t)X_{j+2}|_{\gamma(t)}\mod\mathscr{F}^{j}_{\gamma}(t)
\end{equation*}
and
\begin{equation*}
\mathcal{L}^{j}_{\mathsf{T}}(X_2)|_{\gamma(t)}=h^j_1(t)X_{j+2}|_{\gamma(t)}\mod\mathscr{F}^{j}_{\gamma}(t),
\end{equation*}
where
\[
\mathscr{F}^{j}_{\gamma}(t)=\mathrm{span}\{X_1,\ldots,X_{j+1}\},
\]
and $1\leq j\leq n-2$. Therefore, $\mathscr{F}^{n-1}_{\gamma}(t)=T_{\gamma(t)}\R^n$,  and the geodesic growth vector for the case $h_1(t)\neq 0$  is given by
\[
\mathcal{G}_{\gamma}(t)= (2,3,4,\ldots,j,j+1,\ldots,n). 
\]

\item Assume that $h_1(t)=0$.  By equation (\ref{l1X1}) we have that 
\[
\mathscr{F}^{2}_{\gamma}(t)=\mathrm{span}\{X_1,X_2,X_3\}.
\]
In this case we obtain from the geodesic equations (\ref{vereqg}) the following formulas:
\begin{eqnarray*}
\label{LX1}
\mathcal{L}^{2j}_{\mathsf{T}}(X_1)|_{\gamma(t)}&=&(-1)^{j}h_1(t)h^j_2(t)h^{j-1}_3(t)X_{j+3}|_{\gamma(t)}\mod\mathscr{F}^{2j}_{\gamma}(t)\\ 
\mathcal{L}^{2j+1}_{\mathsf{T}}(X_1)|_{\gamma(t)}&=&(-1)^{j+1}h^{j+1}_2(t)h^{j}_3(t)X_{j+3}|_{\gamma(t)}\mod\mathscr{F}^{2j+1}_{\gamma}(t),\end{eqnarray*}
and
\begin{eqnarray*}
\mathcal{L}^{2j}_{\mathsf{T}}(X_2)|_{\gamma(t)}&=&(-1)^{j}h^j_2(t)h^{j}_3(t)X_{j+2}|_{\gamma(t)}\mod\mathscr{F}^{2j}_{\gamma}(t)\\
\mathcal{L}^{2j+1}_{\mathsf{T}}(X_2)|_{\gamma(t)}&=&(-1)^{j}h_1(t)h^{j}_2(t)h^{j}_3(t)X_{j+3}|_{\gamma(t)} \mod\mathscr{F}^{2j+1}_{\gamma}(t),
\end{eqnarray*}
where 
\begin{eqnarray*}
\mathscr{F}^{2j}_{\gamma}(t)&=&\mathrm{span}\{X_1,\ldots,X_{j+2}\}, \quad \mathscr{F}^{2j+1}_{\gamma}(t)=\mathscr{F}^{2j}_{\gamma}(t),
\end{eqnarray*}
 and $j\geq 1$. From these formulas, which are easily proved by induction, we obtain that the growth vector of the geodesic at times $t$, with $h_1(t)=0$ and $h_3(t)\neq 0$, is
 \[
 \mathcal{G}_{\gamma}(t)=(2,3,3,4,4,\ldots,j,j,\ldots,n).
 \] 
 
\end{enumerate}

Notice that, if there is no time  $\bar{t}$ such that  $h_1(\bar{t})=h_3(\bar{t})=0$ then, by the geodesic equations, the set of times at which $h_1(t)=0$ is discrete, and  $\gamma$ loses equiregularity precisely at these times. Hence, $\gamma$ is equiregular at $t\in [0,T]$ if and only if $h_1(t)\neq0$.  From  (\ref{GrGoursat}) we also have that  the geodesic is ample.

\item Assume that there is a point $\bar{t}\in[0,T]$ such that $h_1(\bar{t})=h_3(\bar{t})=0$. From the geodesic equations (\ref{vereqg}),  the solution $\widetilde{\lambda}(t)$ of $\dot{\widetilde{\lambda}}=\vv{H}(\widetilde{\lambda})$, with initial condition 
\[
\tilde{\lambda}(0)=\left(0,h_2(\bar{t}),0,h_4(\bar{t}), \ldots,h_n(\bar{t})\right),
\]
 is constant $\widetilde{\lambda}(t)\equiv \tilde{\lambda}(0)$. Since the \emph{flag} of the admissible curve $\gamma$ does not depend on the extension $\mathsf{T}$, we can assume that $\mathsf{T}=v_2X_2$, with $v_2\equiv const$. Once again from the geodesic equations (\ref{vereqg}) we obtain that 
\begin{eqnarray*}
\mathcal{L}^j_{\mathsf{T}}(X_1)|_{\gamma(t)}&\in&\mathscr{F}^2_ {\gamma(t)}=\mathrm{span}\ \{X_1,X_2,X_3\}, \\
\mathcal{L}^j_{\mathsf{T}}(X_2)|_{\gamma(t)}&\in&\mathscr{D}_{\gamma(t)},
\end{eqnarray*}
for $j\geq 2$. Therefore, since $n\geq 4$, the geodesic $\gamma$ is not ample, and hence  abnormal.
\end{enumerate}
\end{proof}

Let $\lambda_0$ be the initial covector associated to an ample and equiregular unit-speed geodesic $\gamma$ with Young diagram $D$, which by Theorem \ref{NonVanh} is 
\[
D=\begin{tabular}{ | c | c | c  c |  c | c | }
  \cline{1-2}\cline{5-5}
$a1$   & $a2$ & $\cdots$ & $\cdots$ & $an_a$ \\ \cline{1-2}\cline{5-5}
$bn_b$  & \multicolumn{1}{l}{} & \multicolumn{1}{l}{} & \multicolumn{1}{l}{} & \multicolumn{1}{l}{} \\   \cline{1-1}
\end{tabular}
\]
with $n_a\doteq n-1$ and $n_b\doteq 1$. For such Young diagram, a canonical frame is a smooth family
\[
\{E_{a1},E_{a2},\ldots,E_{an_a},E_{b1},F_{a1},F_{a2},\ldots,F_{an-1},F_{b1}\}\in T_{\lambda_0}\left(T^*\mathbb{R}^n\right),
\]
with the following properties:
\begin{enumerate}
\item It is attached to the Jacobi curve, namely 
\[
J_{\lambda_0}(t)=\mathrm{span}\{E_{a1}(t),E_{a2}(t),\ldots,E_{an_a}(t),E_{b1}(t)\}.
\]
\item From  \cite[Lemma 5.7]{BR2} we have:
\begin{equation}
\label{ChaCF}
E_{b1}(t)=e^{-t\vv{H}}_*\mathfrak{e}=\mathfrak{e}-t\vv{H},
\end{equation}
as a consequence all curvatures $R_{\ast b,\ast 1}$, (where $\ast$ is any other index) vanish.
\item  In this ``easy" Young diagram case, the \emph{normal} condition means that the matrix $[R_{aa,ij}]_{i,j=1,\ldots,n_a}$ is diagonal.
\item They satisfy the structural equations:
\begin{align*}
\dot{E}_{ai}&= E_{a(i-1)}(t)  &i&=2,\ldots,n_a,\\
\dot{E}_{a1}&=- F_{a1}(t) & &\\
\dot{E}_{b1}&=- F_{b1}(t) & &\\
\dot{F}_{ai}&=R_{aa,ii}(t)E_{ai}(t)-F_{a(i+1)}(t), & i&=1,\ldots,n_a-1,\\
\dot{F}_{an_a}&=R_{aa,n_an_a}(t)E_{an_a}(t), & &\\
\dot{F}_{b1}&=0.& &
\end{align*}

\end{enumerate}

We compute the canonical frame following the general algorithm in \cite{ZL}. The computation is presented through a series of lemmas. 
We start by computing  some very useful identities.
\begin{lemma}
\label{IdHh}
The following identities hold true:
\begin{align}
[\vv{H},X_{\theta}]&=-X_3+h_3X_{\bar{\theta}}, \label{HX}\\
[\vv{H},\pa_{\theta}]&=\left\{ 
  \begin{array}{l l}
    X_{\theta}\quad \text{for}\quad n=3,\\
   X_{\theta}+h_2\sum^{n-1}_{i=3}h_{i+1}\pa_{h_i} \quad \text{for}\quad n\geq 4, 
  \end{array} \right. \label{Hht} \\
[\vv{H},\pa_{h_3}]&=- \pa_{\theta}, \label{Hh3}\\
[\vv{H},\pa_{h_i}]&= -h_1\pa_{h_{i-1}}\quad \text{for}\quad  i=4,\ldots,n,\label{Hhi}\\ 
[\vv{H},\mathfrak{e}]&=-\vv{H}.\label{He}
\end{align}
\end{lemma}
\begin{proof}
Let us begin with Eq. (\ref{Hht}) when $n\geq 4$. Recall that $\pa_{\theta}=h_1\pa_{h_2}-h_2\pa_{h_1}$. By Eq. (\ref{ExpH}), we obtain 
\begin{eqnarray*}
[\vv{H},\pa_{\theta}]&=&-[h_1\pa_{h_2}-h_2\pa_{h_1},h_1X_1+h_2X_2]-\pa_{\theta}h_1\sum^{n-1}_{i=3}h_{i+1}\pa_{i}\\
&=& h_2X_1-h_1X_2+h_2\sum^{n-1}_{i=3}h_{i+1}\pa_{i}.
\end{eqnarray*}

For Eq. (\ref{Hh3}) we use once again the explicit expression of $\vv{H}$ of Eq. (\ref{ExpH}) to obtain
\[
[\vv{H},\pa_{h_3}]=[-h_2h_3\pa_{h_1}+h_1h_3\pa_{h_2},\pa_{h_3}]=h_2\pa_{h_1}-h_1\pa_{h_2}=-\pa_{\theta}.
\]
Formula (\ref{Hhi}) follows in similar fashion. For Eq. (\ref{He}), we have
\begin{eqnarray*}
[\vv{H},\mathfrak{e}]&=&[h_1\vv{h}_1+h_2\vv{h}_2,\sum^{n}_{i=1}h_i\pa_{h_i}]\\
&=&-h_1\vv{h}_1-h_2\vv{h}_2+\sum^2_{i=1}\sum^n_{j=1}h_i\vv{h}_i(h_j)\pa_{h_j}+h_ih_j[\vv{h}_i,\pa_{h_j}]\\
&=&-h_1\vv{h}_1-h_2\vv{h}_2.
\end{eqnarray*}
Finally, Eq. (\ref{HX}) follows easily.
\end{proof}

With this lemma in hand we can start computing the \emph{canonical frame}. For convenience we employ the following notation:
\[
f^{(j)}(t)\doteq \frac{d^j}{dt^j}f(t).
\]
\begin{lemma}
\label{ForE}
$E_{an_a}(t)$ is uniquely specified (up to a sign) by the following conditions
\begin{enumerate}
\item $E_{an_a}(t)\in J_{\lambda_0}(t)$,
\item $E^{(i)}_{an_a}(t)\in J_{\lambda_0}(t)$, for $i=1,\ldots,n_a-1$,
\item $\sigma_{\lambda}\left(E^{(n-1)}_{an_a},E^{(n-2)}_{an_a}\right)=1$,
\end{enumerate}
and, by choosing the positive sign, is given by
\[
E_{an_a}(t)=e^{-t\vv{H}}_*h^{2-n_a}_1\pa_{h_n}.
\]
\end{lemma}
\begin{proof}
Condition (1) and the definition of Jacobi curve $J_{\lambda_0}(t)=e^{-t\vv{H}}_{\ast}\mathcal{V}_{\ld(t)}$ imply that
\[
E_{an_a}(t)=e^{-t\vv{H}}_\ast\left(\sum^{n}_{i=1}f_{i}(t)\pa_{h_i}\right),
\]
for some smooth functions $f_i(t)$, with $i=1,\ldots,n$. 
Consider for $i=1,2,\ldots,n,$ the vector fields $V_i(t)\doteq e^{-t\vv{H}}_{\ast}f_i(t)\pa_{h_i}$ and the vector spaces $\mathcal{V}_{i}\doteq\text{span}\{\pa_{h_{i}},\ldots,\pa_{h_{n}}\}$. It is easy to see that for $i\geq 4$  and $0\leq j\leq i-3$, we can write by Eq.(\ref{Hhi}) in Lemma \ref{IdHh}
\begin{equation}
\label{ForV}
V^{(j)}_i(t)=e^{-t\vv{H}}_{\ast}\left((-1)^{j}f_i(t)h^{j}_{1}(t)\pa_{h_{i-j}}\mod\mathcal{V}_{i-j+1}\right).
\end{equation}
 
Condition (2) is equivalent to $\pi_{\ast}\circ e^{t\vv{H}}_{\ast}\dot{E}_{an_a}(t)=0$. Since the vector fields $\pa_{h_1},\ldots,\pa_{h_n}$  are vertical, namely $\pi_{\ast}\pa_{h_l}=0$ for $l=1,\ldots,n$, we obtain the following two equations:
\begin{eqnarray*}
0=\pi_{\ast}\circ e^{t\vv{H}}_{\ast}\dot{V}_{1}(t)=-f_1X_1 \quad\text{and}\quad 0=\pi_{\ast}\circ e^{t\vv{H}}_{\ast}\dot{V}_{2}(t)=-f_2X_2.
\end{eqnarray*}
From this it immediately follows that $f_1=f_2\equiv 0$. Hence, $E_{an_a}(t)=\sum^{n}_{i=3}V_i(t)$. If we use Eq. (\ref{ForV}) with $j=i-3$ for $4\leq i\leq n-1$ together with the identities in Lemma \ref{IdHh}, we get 
\[
\frac{d^2}{dt^2}V^{(i-3)}_i(t)=e^{-t\vv{H}}_{\ast}\left((-1)^{i-2}f_i(t)h^{i-3}_{1}(t)X_{\theta}\mod\mathcal{V}_{1}\right).
\]
Hence, Condition (2) implies that  $E_{an_a}=V_{n}=f_{n}\pa_{h_n}$, for some function $f_n$.  

From Eq. (\ref{ForV}) with $i=n$ and $j=n-4$, we have after some computations
\[
E^{(n-2)}_{an_a}(t)=e^{-t\vv{H}}_{\ast}\left((-1)^{n-2}f_n(t)h^{n-3}_{1}(t)\pa_{\theta}\mod\mathcal{V}_{3}\right),
\]
and
\[
E^{(n-1)}_{an_a}(t)=e^{-t\vv{H}}_{\ast}\left((-1)^{n-2}f_n(t)h^{n-3}_{1}(t)X_{\theta}\mod\mathcal{V}_{1}\right).
\]
We rewrite Condition (3) as
\begin{eqnarray*}
1=\sigma_{\lambda_0}\left(E^{(n-2)}_{an_a}(t),-E^{(n-1)}_{an_a}(t)\right)&=&\sigma_{\ld(t)}\left(e^{t\vv{H}}_{\ast}E^{(n-2)}_{an_a}(t),-e^{t\vv{H}}_{\ast}E^{(n-1)}_{an_a}(t)\right)\\
&=&-f^2h^{2n-6}_1\sigma_{\ld(t)}\left(\pa_{\theta},X_{\theta}\right)\\
&=&f^2h^{2n-6}_1.
\end{eqnarray*}
 By taking the positive sign, we obtain
\[
E_{an_a}(t)=e^{-t\vv{H}}_{\ast}h^{2-n_a}_1\pa_{h_n}.
\]
\end{proof}
The following proposition is an extension of Eq. (\ref{ForV}).
\begin{proposition}Let $E^{(i)}_{an_a}(t)=e^{-t\vv{H}}_{\ast}\left(\sum^{i}_{j=0}a_{ij}(t)\pa_{h_{n-j}}\right)$, for $i=0,\ldots,n-3$, with $n\geq 3$. We have the following formulas for the coefficients $a_{ij}$'s:
\label{aij}
\begin{eqnarray*}
a_{ii}&=&(-1)^{i}h^{2-n_a+i}_{1},\\
a_{ii-1}&=&(-1)^{i-1}\sum^{i-1}_{k=0}h^{k}_1\vv{H}\left(h^{2-n_a+(i-1)-k}_1\right),\quad i\geq 1,\\
a_{ii-2}&=&(-1)^{i-2}\sum^{i-2}_{k=0}h^{k}_1\vv{H}\left(\sum^{i-2-k}_{l_k=0}h^{l_k}_1\vv{H}\left(h^{2-n_a+(i-2)-k-l_k}_1\right)\right),\quad i\geq 2.
\end{eqnarray*}
\end{proposition}
\begin{proof}
It is easy to see that the formulas are valid for $i=0,1,2$. Assume that the formulas are valid for some $2\leq i<n-3$. Now, if we take the derivative of $E^{(i)}_{an_a}=\sum^{i}_{j=0}a_{ij}\pa_{h_{n-j}}$ we have
\begin{eqnarray*}
E^{(i+1)}_{an_a}(t)=e^{-t\vv{H}}_{\ast}\left(-h_1a_{ii}\pa_{n-i-1}+\left(\dot{a}_{ii}-h_1a_{ii-1}\right)\pa_{n-i}+\left(\dot{a}_{ii-1}-h_1a_{ii-2}\right)\pa_{n-i+1}\mod\mathcal{V}_{n-i+2}\right).
\end{eqnarray*}
where $\mathcal{V}_{n-i+2}=\text{span}\{\pa_{h_{n-i+2}},\ldots,\pa_{h_{n}}\}$. Hence,
\begin{eqnarray*}
a_{i+1i+1}&=&-h_1a_{ii}\\
&=&(-1)^{i+1}h^{2-n_a+i+1}_{1};\\
a_{i+1i}&=&\dot{a}_{ii}-h_1a_{ii-1}\\
&=&(-1)^{i}\vv{H}\left(h^{2-n_a+i}_{1}\right)+(-1)^{i}\sum^{i-1}_{k=0}h^{k+1}_1\vv{H}\left(h^{2-n_a+(i-1)-k}_1\right)\\
&=&(-1)^{i}\vv{H}\left(h^{2-n_a+i}_{1}\right)+(-1)^{i}\sum^{i}_{l=1}h^{l}_1\vv{H}\left(h^{2-n_a+i-l}_1\right) \quad (l=k+1)\\
&=&(-1)^{i}\sum^{i}_{l=0}h^{l}_1\vv{H}\left(h^{2-n_a+i-l}_1\right) .
\end{eqnarray*}
For the last term we have
\begin{eqnarray*}
a_{i+1i-1}&=&\dot{a}_{ii-1}-h_1a_{ii-2}\\
&=&(-1)^{i-1}\vv{H}\left(\sum^{i-1}_{l=0}h^{l}_1\vv{H}\left(h^{2-n_a+(i-1)-l}_1\right)\right)\\
&&+(-1)^{i-1}\sum^{i-2}_{k=0}h^{k+1}_1\vv{H}\left(\sum^{i-2-k}_{l_k=0}h^{l_k}_1\vv{H}\left(h^{2-n_a+(i-2)-k-l_k}_1\right)\right)\\
&=&(-1)^{i-1}\vv{H}\left(\sum^{i-1}_{l=0}h^{l}_1\vv{H}\left(h^{2-n_a+(i-1)-l}_1\right)\right)\\
&&+(-1)^{i-1}\sum^{i-1}_{j=1}h^{j}_1\vv{H}\left(\sum^{i-1-j}_{l_j=0}h^{l_j}_1\vv{H}\left(h^{2-n_a+(i-1)-j-l_j}_1\right)\right)\quad(j=k+1)\\
&=&(-1)^{i-1}\sum^{i-1}_{j=0}h^{j}_1\vv{H}\left(\sum^{i-1-j}_{l_j=0}h^{l_j}_1\vv{H}\left(h^{2-n_a+(i-1)-j-l_j}_1\right)\right).
\end{eqnarray*}
\end{proof}
For future computations  we need explicit formulas for $F_{a1}(0)$ and its derivative $\dot{F}_{a1}(0)$. By Proposition \ref{aij}, with $i=n-3$, we have:
\begin{itemize} 
\item For $n=3$,
\[
E_{an_a}(0)=\pa_{h_3}.
\]
\item For $n=4$,
\[
\dot{E}_{an_a}(0)=-\pa_{h_3}+\frac{(n-2)(n-3)}{2}\frac{h_2h_3}{h^2_1}\pa_{h_4}.
\]
\item For $n\geq 5$,
\begin{eqnarray*}
E^{(n-3)}_{an_a}(0)&=&(-1)^{n-4}\left(-\pa_{h_3}+\sum^{n-4}_{k=0}h^{k}_1\vv{H}\left(h^{-1-k}_1\right)\pa_{h_4}\right)\\
&&+(-1)^{n-5}\left(\left(\sum^{n-5}_{k=0}h^{k}_1\vv{H}\left(\sum^{n-5-k}_{l_k=0}h^{l_k}_1\vv{H}\left(h^{-2-k-l_k}_1\right)\right)\right)\pa_{h_5}\mod\mathcal{V}_6\right)\\
&=&(-1)^{n-4}\left(-\pa_{h_3}+\frac{(n-3)(n-2)}{2}\frac{h_2h_3}{h^2_1}\pa_{h_4}\right)\\
&&+(-1)^{n-5}\left(\left(A_1(n)\frac{h^2_2h^2_3}{h^4_1}+A_2(n)\frac{h^2_3+h_2h_4}{h^2_1}\right)\pa_{h_5}\mod\mathcal{V}_6\right),
\end{eqnarray*}
where $A_{1}(n)$ and $A_2(n)$ are given by 
\begin{eqnarray}
A_1(n)&=&\sum^{n-5}_{k=0}(3+k)\left(\sum^{n-5-k}_{j_k=0}(k+j_k+2)\right)\nonumber\\
&=&\frac{1}{8}\left(n-4\right)\left(\left(n-4\right)\left(n^2+10n-27\right)-\left(n-5\right)\left(8n-18\right)\right)\nonumber\\
&=&\frac{1}{8}(n+3)(n-2)(n-3)\left(n-4\right)\label{A1},\\
A_2(n)&=&\sum^{n-5}_{k=0}\left(\sum^{n-5-k}_{j_k=0}(k+j_k+2)\right)\nonumber\\
&=&\frac{1}{6}\left(n-4\right)\left(3\left(n-4\right)\left(n-1\right)-n\left(n-5\right)\right)\nonumber\\
&=&\frac{1}{3}(n-2)(n-3)\left(n-4\right)\label{A2}.
\end{eqnarray}
\end{itemize}

After some  simple computations using the identities in Lemma \ref{IdHh} we obtain the following formulas for $n\geq 3$,
\begin{eqnarray}
\label{Fa1}
&&\\
F_{a1}(0)&=&(-1)^{n-1}\left(X_{\theta}+\frac{(n-2)(n-3)}{2}\frac{h_2h_3}{h_1}\pa_{\theta}\right)\nonumber\\
&&+(-1)^{n-2}\left(\left(\frac{(n-2)(n-3)}{2}\left(3\frac{h^2_2h^2_3}{h^2_1}+2h^2_3+2h_2h_4\right)-h_2h_4\right)\pa_{h_3}\right)\nonumber\\
&&+(-1)^{n-2}\left(\left(A_1(n)\frac{h^2_2h^2_3}{h^2_1}+A_2(n)\left(h^2_3+h_2h_4\right)\right)\pa_{h_3}\mod\mathcal{V}_4\right)\nonumber,
\end{eqnarray}
and
\begin{eqnarray}
\label{dFa1}
&&\\
\dot{F}_{a1}(0)&=&(-1)^{n-1}\left(-X_3+h_3X_{\bar{\theta}}+\frac{(n-2)(n-3)}{2}\frac{h_2h_3}{h_1}X_{\theta}\right)\nonumber\\
&&+(-1)^{n-1}\left(\left(\frac{(n-2)(n-3)}{2}\left(4\frac{h^2_2h^2_3}{h^2_1}+3h^2_3+3h_2h_4\right)-h_2h_4\right)\pa_{\theta}\right)\nonumber\\
&&+(-1)^{n-1}\left(\left(A_1(n)\frac{h^2_2h^2_3}{h^2_1}+A_2(n)\left(h^2_3+h_2h_4\right)\right)\pa_{\theta}\mod\mathcal{V}_3\right)\nonumber,
\end{eqnarray}
with the convention $h_4\equiv0$ for $n=3$. Notice that for $n=3,4$, the functions $A_1,A_2$ are defined and, moreover, they vanish.

 \begin{proof}[Proof of Theorem \ref{gasym}] By direct inspection, the orthonormal basis $\{X_a,X_b\}$ for $\mathscr{D}_{x_0}$ obtained by projection of the canonical frame is
 \[
 X_a\doteq\pi_{\ast}F_{a1}(0) \quad  X_b\doteq\pi_{\ast}F_{b1}(0). 
 \]
 Recall that in the coordinates associated to the splitting $\Sigma_{\lambda}=\mathcal{V}_{\lambda_0}\oplus\mathcal{H}_{\lambda_0}$ we have 
\[
\mathcal{Q}_{\lambda_0}(t)=\frac{d}{dt}S^{\flat}(t)^{-1}.
\]
From Eqs. (\ref{fSbeta1}), (\ref{fSbeta2}) in  Theorem \ref{Sbeta},  the formula for $S^{\flat}(t)^{-1}$ in the basis $\{X_a,X_b\}$ is given by 
\[
S^{\flat}(t)^{-1}=-\frac{1}{t}
\begin{pmatrix}
(n-1)^2 & 0\\
0& 1
\end{pmatrix}
+\frac{t}{3}
\begin{pmatrix}
\frac{3(n-1)}{4(n-1)^2-1}R_{aa,11}(0) & 0\\
0 & 0
\end{pmatrix}
+O\left(t^2\right),
\]
 since $R_{ab,11}=R_{ba,11}=R_{bb,11}\equiv0$ by  \cite[Lemma 5.7]{BR2}. Therefore, the curvature operator has the following expression
\begin{equation}
\label{Goursat-CO}
\mathcal{R}_{\lambda_0}=\frac{3(n-1)}{4(n-1)^2-1}
\begin{pmatrix}
R_{aa,11}(0)&0\\
0&0
\end{pmatrix}.
\end{equation}
From Eq. (\ref{Fa1}) and Eq. (\ref{dFa1}) we obtain the following explicit expression for the term $R_{aa,11}$ with $n\geq 3$,
\begin{eqnarray}
\label{CurR1}
R_{aa,11}(0)&=&\sigma_{\lambda}(\dot{F}_{a1}(0),F_{a1}(0)) \\
&=&\frac{1}{2}\left(2-5(n-2)(n-3)-4A_2(n)\right)\left(h^2_3+h_2h_4\right)\nonumber\\
&&+\frac{1}{4}\left((n-2)^2(n-3)^2-14(n-2)(n-3)-8A_1(n)\right)\frac{h^2_2h^2_3}{h^2_1}\nonumber\\
&=&-\frac{1}{6}(n-1)\left(12+n\left(4n-17\right)\right)\left(h^2_3+h_2h_4\right)\nonumber\\
&&-(n-1)(n-2)(n-3)h^2_3\tan^2\left(\theta+\frac{\pi}{2}\right),\nonumber
\end{eqnarray}
where $A_{1}(n)$ and $A_{2}(n)$ are given by Eq. (\ref{A1}) and Eq. (\ref{A2}) for $n\geq 3$. Here we use again the convention $h_4\equiv0$ for $n=3$. 
\end{proof}

We show next how the complete set of invariants, $R(t)$, in principle, might be obtained from formulas (\ref{Fa1}), (\ref{dFa1}) and (\ref{CurR1}). Since the matrix $R(t)$ is normal,  the sub-matrix $R_{aa}=[R_{aa,ij}]$ is diagonal, and hence,  from the structural equations we easily get
 \begin{eqnarray}
\label{Fai}
 F_{ai}&=&\sum^{i-1}_{j=1}(-1)^{j-1}\frac{d^{j-1}}{dt^{j-1}}\left(R_{i-ji-j}E_{ai-j}\right)+(-1)^{i-1}F^{(i-1)}_{a1}.
  \end{eqnarray}
Therefore, using Eq. (\ref{Fai}) and the fact that $R_{aa,ii}=\sigma(\dot{F}_{ai},F_{ai})$, we obtain for $1\leq i\leq n-1$,
 \begin{eqnarray*}
 \label{Rai}
R_{aa,ii}(t)&=&\sum^{i-1}_{j=1}(-1)^{i+j-2}\sigma_{\lambda_0}\left(\frac{d^{j}}{dt^{j}}\left(R_{aa,i-ji-j}E_{ai-j}\right)(t),F^{(i-1)}_{a1}(t)\right)\\
&&+\sum^{i-1}_{k,l=1}(-1)^{k+l-2}\sigma_{\lambda_0}\left(\frac{d^{k}}{dt^{k}}\left(R_{aa,i-ki-k}E_{ai-k}\right)(t),\frac{d^{l-1}}{dt^{l-1}}\left(R_{aa,i-li-l}E_{ai-l}\right)(t)\right)\nonumber\\
&&+\sigma_{\lambda_0}\left(F^{(i)}_{a1}(t),F^{(i-1)}_{a1}(t)\right).
 \end{eqnarray*}

\section{Engel group}
\label{Engel}

 In this section we analyze the  geodesics in the Engel group, $n=4$. In order to do this we limit ourselves to the level surface $\{H=\frac{1}{2}\}$ and consider the coordinate system $(\theta,c,\alpha)$:
 \[
 h_1=\cos\left(\theta+\frac{\pi}{2}\right), \quad h_2=\sin\left(\theta+\frac{\pi}{2}\right), \quad h_3=c, \quad h_4=\alpha.
 \]
In the variables $(\theta,c,\alpha)$ the Hamiltonian system assumes the following form:
\begin{eqnarray}
\dot{\theta}&=&c,\nonumber\\
\dot{c}&=&-\alpha\sin\theta, \label{hengel}\\
\dot{\alpha}&=&0\nonumber.
\end{eqnarray}
Note that this system for the costate variables reduces to the pendulum equation
\begin{equation}
\label{Pendulum}
\ddot{\theta}=-\alpha\sin\theta, \quad \dot{\alpha}=0.
\end{equation}
Let us introduce the \emph{energy integral} of the pendulum (\ref{Pendulum}):
\[
E = \frac{c^2}{2} - \alpha \cos\theta \in [-|\alpha|, +\infty),\quad  \dot{E} =  0.
\]

Let $\gamma:[0,T]\rightarrow \R^4$ be a normal geodesic and $\mathsf{T}$ be an admissible extension of $\gamma$. From Theorem \ref{NonVanh} we get that if for the initial covector $\lambda$ we have $h_1(\lambda)=h_3(\lambda)=0$, then the geodesic is not ample at $t=0$, and hence not ample for all $t$. 
On the other hand, from Theorem \ref{NonVanh} we also have that if there is no time $t$ such that $h_1(\lambda(t))$ and $h_3(\lambda(t))$ are both zero, the geodesic growth vector is
\begin{equation}
\mathcal{G}_{\gamma}(t)=\left\{ 
  \begin{array}{l l}
    (2,3,4) & \quad \text{if $h_1(\lambda(t))\neq 0$},\\
   (2,3,3,4) & \quad \text{if $h_1(\lambda(t))= 0$}. 
  \end{array} \right.
\label{Grow-E}
\end{equation}

We will use extensively the results in   \cite{ArSac} to  study   the set of times of loss of equiregularity of a geodesic. The family of normal extremal trajectories can be parametrized by points in the cylinder
\begin{eqnarray*}
C&=&T^{\ast}_{x_0}\mathbb{R}^4\cap \{H=\frac{1}{2}\} =\{(h_1,h_2,h_3,h_4)\in \R^4 : h^2_1+h^2_2=1\}\\
&=& \{(\theta,c,\alpha) : \theta \in S^1, c,\alpha \in \R\}.
\end{eqnarray*}
Following \cite{ArSac}, we partition $C$ into subsets corresponding to different types of pendulum trajectories:
\[
C= \bigcup^{7}_{i=1} C_i,\quad C_i\cap C_j =\emptyset,\  i\neq j,\quad \lambda=(\theta,c,\alpha),
\]
\begin{eqnarray*}
C_1 &=&\{\lambda \in C: \alpha\neq 0, E\in (-|\alpha|,|\alpha|)\}, \\
C_2 &=&\{\lambda \in C:\alpha\neq 0,E\in(|\alpha|,+\infty)\}, \\
C_3 &=&\{\lambda \in C:\alpha\neq0,E=|\alpha|,c\neq0\}, \\
C_4 &=&\{\lambda \in C:\alpha\neq0,E=-|\alpha|\},\\
C_5 &=&\{\lambda \in C:\alpha\neq0,E=|\alpha|,c=0\}, \\
C_6 &=&\{\lambda \in C:\alpha=0,c\neq 0\},\\
C_7 &=&\{\lambda \in C:\alpha=c=0\}.
\end{eqnarray*}

The sets $C_i$, $i = 1,\ldots, 5,$ are further subdivided into subsets depending on the sign of $\alpha$:
\[
C^+_i=C_i\cup\{\alpha>0\},\quad C^-_i=C_i\cup\{\alpha<0\},\quad i \in \{1,\ldots, 5\}.
\]

 Here and throughout, sn, cn, and dn  are Jacobian elliptic functions (see \cite{WhW}). Following \cite{ArSac}, in order to calculate the extremal trajectories from the subsets $C_1,C_2$ and $C_3$ we introduce coordinates $(\varphi,k,\alpha)$ in which the system (\ref{hengel}) is straightened out. Since the general case $\alpha\neq 0$ can be obtained from the special case $\alpha>0$,  see \cite{ArSac}, we will describe these coordinates assuming that $\alpha>0$.
 
 In the domain $C^+_1$, we set
 \[
 k=\sqrt{\frac{E+\alpha}{2\alpha}}=\sqrt{\frac{c^2}{4\alpha}+\sin^2\frac{\theta}{2}}\in (0,1),
 \]
 \[
 \sin\frac{\theta}{2}=k\ \mathrm{sn}\ (\sqrt{\alpha}\varphi),\quad \cos\frac{\theta}{2}=\mathrm{dn}\ (\sqrt{\alpha}\varphi),\quad \frac{c}{2}=k\sqrt{\alpha}\ \mathrm{cn}\ (\sqrt{\alpha}\varphi),\quad \varphi\in [0,4K],
 \]
where $4K$ is the period of the Jacobian elliptic functions sn and cn. 

In the domain $C^+_2$, we set
\[
k=\sqrt{\frac{2\alpha}{E+\alpha}}=\frac{1}{\sqrt{c^2/(4\alpha)+\sin^2(\theta/2)}}\in (0,1),
\]
\[
\sin\frac{\theta}{2}=\mathrm{sgn}\ c  \ \mathrm{sn}\ \frac{\sqrt{\alpha}\varphi}{k}, \quad \cos\frac{\theta}{2}=\mathrm{cn}\ \frac{\sqrt{\alpha}\varphi}{k},\quad \frac{c}{2}=\mathrm{sgn}\ c \frac{\sqrt{\alpha}}{k}\ \mathrm{dn}\ \frac{\sqrt{\alpha}\varphi}{k}, 
\]
where $ \varphi\in [0,2kK]$.

In the domain $C^+_3$, we set
\[
k=1,
\]
\[
\sin\frac{\theta}{2}=\mathrm{sgn}\ c\ \tanh (\sqrt{\alpha}\varphi),\quad \cos\frac{\theta}{2}=\frac{1}{\cosh (\sqrt{\alpha}\varphi)},\quad \frac{c}{2}=\mathrm{sgn}\ c\ \frac{\sqrt{\alpha}}{\cosh (\sqrt{\alpha}\varphi)},
\]
and $\varphi\in (-\infty,\infty)$.

Immediate diferentiation shows that in these coordinates the subsystem for the
costate variables (\ref{hengel}) takes the following form:
\[
\dot{\varphi} = 1, \quad  \dot{k} = 0,\quad \dot{\alpha}=0.
\]
so that it has solutions
\[
\varphi(t)=\varphi_t =\varphi+t,\quad k=\mathrm{const},\quad \alpha=\mathrm{const}.
\]
The elliptic coordinate $\varphi$ is the time of movement along trajectories of the pendulum equation and $k$ is a parameter that distinguishes trajectories with different energies.

\begin{proposition}
For   $\lambda \in C_1,C_2, C_6$ the geodesic is ample and has an infinite and discrete set of times of loss of equiregularity.  If   $\lambda\in C_3$, the geodesic is ample and has an unique time of loss of equiregularity. For   $\lambda\in C_4, C_5$ the geodesic is not ample and hence abnormal. If   $\lambda\in C_7$ then it depends: usually the geodesic is ample and equiregular, but there are some non-ample cases (corresponding to straight lines).
\end{proposition}
\begin{proof}
We have, by the symmetries of the Hamiltonian system (\ref{hengel}), namely dilations and reflections, that for $\lambda\in C_1,C_2,C_3$,  the solution to the system (\ref{hengel}) for the case $\alpha\neq0$ can be recovered from the special case $\alpha=1$, see \cite{ArSac} for details. Hence, things do not qualitatively change if we assume that for $\lambda\in C_1,C_2,C_3$ we have $\alpha=1$.
\begin{enumerate}
\item Let $\lambda\in C_1$ with $\alpha=1$.   In this case we have
\[
h_1(t)=-2k\ \mathrm{sn}\ \varphi_t\ \mathrm{dn}\ \varphi_t, \quad k=\sqrt{\frac{E+1}{2}},
\]
and $\varphi_t=\varphi+t$, with $\varphi\in[0,4K]$.  Hence, the geodesic has a infinite and discrete set of times of loss of equiregularity. 
\item  For $\lambda\in C_2$ with $\alpha=1$,  we have
\[
h_1(t)=-2\ \mathrm{sgn}\ c\ \mathrm{sn}\ \psi_t\ \mathrm{cn}\ \psi_t, \quad k=\sqrt{\frac{2}{E+1}},
\]
and $\psi_t=\frac{\varphi+t}{k}$, with $\varphi\in[0,2kK]$. As in the previous case, the geodesic has a infinite and discrete set of times of loss of equiregularity.
\item Assume that $\lambda\in C_3$ and $\alpha=1$. Then, we have
\[
h_1(t)=-2\ \mathrm{sgn}\ c\ \frac{\tanh\varphi_t}{\cosh\varphi_t}, 
\]
where $\varphi_t=\varphi+t$ and $\varphi\in\R$. Hence, $h_1(t)=0$ for $\varphi+t=0$.

\item Suppose that $\lambda$ is in $C_4$ or in $C_5$. In either case, we will have that $h_1(0)=h_3(0)=0$. Therefore, the geodesic is not ample, and hence abnormal. 

\item Let $\lambda\in C_6$. By the geodesic equations we have that $\ddot{\theta}=0$, so that $\theta(t)=ct+\theta$, where $c=\mathrm{const}\neq0$ and $\theta=\mathrm{const}$. Therefore, 
\[
h_1(t)=\cos(\theta(t))=0 \quad \text{if and only if} \quad ct+\theta= (2k+1)\frac{\pi}{2},
\]
where $k\in \mathbb{Z}$.
\item Let $\lambda\in C_7$. In this case we also have  $\dot{\theta}=0$, so  $\theta(t)\equiv\theta$, where $\theta=\mathrm{const}$. Therefore, if $\theta\neq (2k+1)\frac{\pi}{2}$, with $k\in \mathbb{Z}$, then $h_1(t)\equiv\mathrm{const}\neq 0$. Moreover, if $\theta=(2k+1)\frac{\pi}{2}$, then 
$h_1(t)=h_3(t)\equiv 0$, so the geodesic is abnormal.
\end{enumerate}

\end{proof}

We close this section by noticing that if we use the \emph{energy integral} and Eq. (\ref{CurR1}) we immediately obtain the following bound:
\[
R_{aa,11}(\lambda)=-6c^2\csc^2\theta+4E\leq 4E.
\]

\section{Cartan group} 
\label{Sec:Cartan}

We now turn our attention to the \emph{Cartan group} and compute its \emph{curvature operator}. In $\R^5$ with coordinates $\left(x,y,z,v,w\right)$,  we consider the vector fields
 \[
 X_1=\frac{\pa}{\pa x}-\frac{y}{2}\frac{\pa}{\pa z}-\frac{x^2+y^2}{2}\frac{\pa}{\pa w}, \quad X_2= \frac{\pa}{\pa y}-\frac{x}{2}\frac{\pa}{\pa z}+\frac{x^2+y^2}{2}\frac{\pa}{\pa v}. 
 \]
 
 The vector fields $X_1$ and $X_2$ are called \emph{horizontal} and satisfy the H\"ormander condition, i.e. they generate the whole tangent bundle by their commutators: 
\begin{eqnarray*}
X_3&=&\left[X_1,X_2\right]=\frac{\pa}{\pa z}+x\frac{\pa}{\pa v}+y\frac{\pa}{\pa w},\\
X_4&=&\left[X_1,X_3\right]=\frac{\pa}{\pa v},\\
X_5&=&\left[X_2,X_3\right]=\frac{\pa}{\pa w},
\end{eqnarray*} 
and all other commutators being  zero. Notice that the distribution of the Cartan group is not of Goursat-type.

We can introduce a unique Lie group structure (law of multiplication) in $\R^5$, making a Lie group of $\R^5$,  so that $X_1, \ldots, X_5,$ become basic left-invariant  fields on this Lie group. We denote this Lie group by $\mathfrak{C}$.  We have that $\mathscr{D}|_h\doteq\mathrm{span}\ \{X_1,X_2\}|_h$, for all $h\in \mathfrak{C}$, is a distribution left-invariant by the action $L_g:\mathfrak{C}\rightarrow\mathfrak{C}$. Any left-invariant scalar product $g$ on $\mathscr{D}|_h$ induces a left-invariant sub-Riemannian structure $(\mathscr{D},g)$ on $\mathfrak{C}$. Since any two different choices give rise to isometric sub-Riemannian structures,  we choose without loss of generality $g$ such that $X_1$ and $X_2$ are orthonormal.

Let $\{\nu_1,\ldots,\nu_5\}$  be the dual frame of $\{X_1,\ldots, X_5\}$. This dual frame induces coordinates $\left\{h_1,\ldots,h_{5}\right\}$ in each fiber of $T^*\mathbb{R}^5$,
\[
\ld=\left( h_1,\ldots,h_5\right)\quad\Longleftrightarrow\quad \ld=h_1\nu_1+\ldots+h_5\nu_5
\]
where $h_i(\ld)=\left\langle \ld,X_i\right\rangle$ are the linear-on-fibers functions associated with $X_i$, for $i=1,\ldots,5$. 

Let $\vv{h}_i\in \text{Vec}(T^*\mathbb{R}^5)$ be the Hamiltonian vector fields associated with $h_i\in C^{\infty}(T^*\mathbb{R}^5)$ for $i=1,\ldots,5$, respectively. Consider the vertical vector fields $\pa_{h_i}\in \text{Vec}(T^*\mathbb{R}^5)$, for $i=1,\ldots,5$. The vector fields 
\[
\vv{h}_1,\ldots,\vv{h}_5,\pa_{h_1},\ldots,\pa_{h_5},
\]
are a local frame of vector fields of $T^*\mathbb{R}^5$. Equivalently, we can introduce the cylindrical coordinates $\rho,\theta,h_3,h_4,h_5$ on each fiber of $T^*\mathbb{R}^5$ by 
\[
h_1=\rho\cos\theta,\quad h_2=\rho\sin\theta,
\]
with $\rho\in (0,+\infty)$ and $\theta \in (-\pi,\pi]$, and employ instead the local frame
\[
\vv{h}_1,\ldots,\vv{h}_5,\pa_{\theta}, \pa_{\rho},\pa_{h_3},\pa_{h_4},\pa_{h_5},
\]
where  $\pa_{\theta}\doteq h_1\pa_{h_2}-h_2\pa_{h_1}$.

The \emph{Euler vector field} is given by
\[
\mathfrak{e}\doteq \sum^5_{i=1}h_i\pa_{h_i}=\rho\pa_{\rho}+\sum^{5}_{i=3}h_i\pa_{h_i}.
\]
Recall that $\mathfrak{e}$ is a vertical vector field on $T^*\mathbb{R}^5$, i.e. $\pi_*\mathfrak{e}=0$, and is the generator of the dilations  $\ld\mapsto e^c\ld$ along the fibers of $T^*\mathbb{R}^5$. The sub-Riemannian Hamiltonian is 
\[
H=\frac{1}{2}\left(h^2_1+h^2_2\right),
\]
and, therefore, the Hamiltonian vector field $\vv{H}$ is given by
\[
\vv{H}=h_1\vv{h}_1+h_2\vv{h}_2.
\]
Let $\{\nu_{\theta},\nu_{\bar{\theta}}\}$ be the dual co-frame associated with the frame $\{X_{\theta},X_{\bar{\theta}}\}$, given by $X_{\bar{\theta}}\doteq h_1X_1+h_2X_2$ and $X_{\theta}\doteq h_2X_1-h_1X_2$. The symplectic form in the cylindric coordinates $\rho,\theta,h_3,h_4,h_5$ has the following expression:
\begin{eqnarray*}
\sigma&=&\rho d\rho\wedge\nu_{\bar{\theta}}-\rho^2d\theta\wedge\nu_{\theta}+\rho^2 h_3\nu_{\bar{\theta}}\wedge\nu_{\theta}+\left(h_2h_{4}-h_1h_5\right)\nu_{3}\wedge\nu_{\theta}\\
&&+\left(h_1h_4+h_2h_5\right)\nu_{3}\wedge\nu_{\bar{\theta}}+\sum^{5}_{i=3}dh_i\wedge\nu_i.
\end{eqnarray*}
The Hamiltonian vector field $\vv{H}$ can be written as
\[
\vv{H}=X_{\bar{\theta}}+h_3\pa_{\theta}+\left(h_1h_4+h_2h_5\right)\pa_{h_3}.
\]
\subsection{Ample geodesics in the Cartan group} 
Observe that $h_4$ and $h_5$ are \emph{first integrals} of the Hamiltonian system: $\vv{H}h_4=\vv{H}h_5=0$. Another important first integral of the system is the so-called \emph{energy integral}:
\[
E=\frac{h^2_3}{2}+h_1h_5-h_2h_4.
\]
It is not difficult to see that $\vv{H}E=0$. In the coordinates 
\[
h_1=\cos\theta,\quad h_2=\sin\theta, \quad c=h_3, \quad h_4=\alpha\sin\beta, \quad h_5=-\alpha\cos\beta, 
\]
on the fibres of the unit cotangent bundle, where $\rho=1$, the energy has the form
\[
E= \frac{c^2}{2} -\alpha \cos(\theta-\beta)\in [-\alpha,\infty).
\]

\begin{proposition}
\label{NonVanh3}
Let  $\gamma=\pi\circ\lambda:[0,T]\rightarrow \mathbb{R}^5$ be a normal unit-speed  geodesic in $\mathfrak{C}$. Then, 
\begin{enumerate}
\item If  $h_3$ is not identically equal to  zero, then $\gamma$ is equiregular at $t\in [0,T]$ if and only if $h_3(t)\neq0$. Moreover, $\gamma$ is ample and the geodesic growth vector is given by
\begin{equation}
\label{GrCartan}
\mathcal{G}_{\gamma}(t)=\left\{ 
  \begin{array}{l l}
    (2,3,4,5), & \quad \text{if $h_3(t)\neq 0$},\\
   (2,3,4,4,5), & \quad \text{if $h_3(t)= 0$}. 
  \end{array} \right.
\end{equation}
\item If $h_3\equiv 0$, the geodesic is not ample, and we are in the case of an abnormal geodesic. 
\end{enumerate}
\end{proposition}
\begin{proof}
Let $\gamma:[0,T]\rightarrow \R^5$ be a normal geodesic and  $\mathsf{T}=v_1X_1+v_2X_2$ be an admissible extension of $\gamma$. Recall the geodesic equations for the Cartan Group:
\begin{equation}
\label{vereqc}
\dot{h}_1=-h_2h_3, \quad \dot{h}_2=h_1h_3, \quad \dot{h}_3=h_1h_4+h_2h_5,\quad \dot{h}_4=\dot{h}_5=0.
\end{equation}
\begin{enumerate} 
\item Let us compute the growth vector when $h_3$ is not identically equal to zero. For $\mathscr{F}_2$ we have
\begin{eqnarray*}
\mathcal{L}_{\mathsf{T}}(X_1)&=&-h_2X_3\ \mod \mathscr{F}^1,  \label{Lie-C1}\\
\mathcal{L}_{\mathsf{T}}(X_2)&=&h_1X_3\ \mod \mathscr{F}^1. \label{Lie-C2}
\end{eqnarray*}
Hence, we obtain $\mathscr{F}^2=\mathrm{span}\{X_1,X_2,X_3\}$ for all $t$.
Now, observe that
\begin{eqnarray*}
\mathcal{L}^2_{\mathsf{T}}(X_1)&=& -h_2\left(h_1X_4+h_2X_5\right) \mod \mathscr{F}^2\label{Lie-C3}, \\ 
\mathcal{L}^2_{\mathsf{T}}(X_2)&=& h_1\left(h_1X_4+h_2X_5\right) \mod \mathscr{F}^2\label{Lie-C4}.
\end{eqnarray*}
Therefore, $\mathscr{F}_3=\mathrm{span}\{X_1,X_2,X_3,h_1X_4+h_2X_5\}$, and $k_3=\mathrm{dim}\ \mathscr{F}^3=4$, for all $t$.

For the computation of $\mathscr{F}^4$, we have
\begin{eqnarray*}
\mathcal{L}^3_{\mathsf{T}}(X_1)&=& h_2h_3\left(h_2X_4-h_1X_5\right) \mod \mathscr{F}^3, \\ 
\mathcal{L}^3_{\mathsf{T}}(X_2)&=& -h_1h_3\left(h_2X_4-h_1X_5\right) \mod \mathscr{F}^3.
\end{eqnarray*}
Then, if $h_3(t)\neq 0$, we obtain
\[
\mathscr{F}^4=\mathrm{span}\{X_1,X_2,X_3,h_1X_4+h_2X_5, h_2X_4-h_1X_5\}
\]
and $k_4=\mathrm{dim}\ \mathscr{F}^4=5$, since $h^2_1+h^2_2=1$.

If $h_3(\bar{t})=0$, then $k_4(\bar{t})=4$. In this case, we need to compute $\mathscr{F}^5$. Notice that 
\begin{eqnarray*}
\mathcal{L}^4_{\mathsf{T}}(X_1)&=&h_2\left(h_1h_4+h_2h_5\right)\left(h_2X_4-h_1X_5\right)\label{Lie-C31}\mod\mathscr{F}^4,\\
\mathcal{L}^4_{\mathsf{T}}(X_1)&=&-h_1\left(h_1h_4+h_2h_5\right)\left(h_2X_4-h_1X_5\right) \mod\mathscr{F}^4.\label{Lie-C41}
\end{eqnarray*}
 
 From the geodesic equations (\ref{vereqc}), if $h_1(\bar{t})h_4(\bar{t})+h_2(\bar{t})h_5(\bar{t})=0$, then $h_3\equiv 0$, which contradicts the hypothesis. Hence, we have that the geodesic growth vector for a geodesic with $h_3$ not identically equal to zero  is given by (\ref{GrCartan}).
 
  \item Now, if $h_3\equiv 0$, then $h_1(t)$ and $h_2(t)$ are constant. Hence, from the computations of case (1), it follows that in this case the geodesic is not ample, and hence the geodesic is abnormal.

\end{enumerate}
\end{proof}

The family of normal extremal trajectories can be parametrized by points in the cylinder
\begin{eqnarray*}
C&=&T^{\ast}_{x_0}\mathbb{R}^5\cap \{H=\frac{1}{2}\} =\{(h_1,h_2,h_3,h_4,h_5)\in \R^5 : h^2_1+h^2_2=1\}\\
&=& \{(\theta,c,\alpha,\beta) : \theta,\beta \in S^1, c,\alpha \in \R\}.
\end{eqnarray*}
Following \cite{Sac1}, we partition $C$ into subsets corresponding to different types of pendulum trajectories:
\[
C= \bigcup^{7}_{i=1} C_i,\quad C_i\cap C_j =\emptyset,\  i\neq j,\quad \lambda=(\theta,c,\alpha,\beta),
\]
\begin{eqnarray*}
C_1 &=&\{\lambda \in C: \alpha\neq 0, E\in (-\alpha,\alpha)\}, \\
C_2 &=&\{\lambda \in C:\alpha\neq 0,E\in(\alpha,+\infty)\}, \\
C_3 &=&\{\lambda \in C:\alpha\neq0,E=\alpha,\theta-\beta\neq\pi\}, \\
C_4 &=&\{\lambda \in C:\alpha\neq0,E=-\alpha\},\\
C_5 &=&\{\lambda \in C:\alpha\neq0,E=\alpha,\theta-\beta=\pi\}, \\
C_6 &=&\{\lambda \in C:\alpha=0,c\neq 0\},\\
C_7 &=&\{\lambda \in C:\alpha=c=0\}.
\end{eqnarray*}

 From Proposition \ref{NonVanh3} we have that ample geodesics belong to $C_1,C_2,C_3$, and $C_6$.  We now analyze the equiregularity of normal geodesics with initial values in  these subsets of the cylinder $C$. We will use extensively the results in \cite{Sac1}.
 
 Following  \cite{Sac1}, we introduce elliptic coordinates $(k,\varphi,\alpha,\beta)$ in the subsets $C_1,C_2$ and $C_3$ of the cylinder $C$  as following.
 
For  $\lambda\in C_1$, we set:
 \begin{gather*}
 \alpha\neq 0, \quad E\in (-\alpha,\alpha),\\
 k=\sqrt{\frac{E+\alpha}{2\alpha}}=\sqrt{\sin^2\frac{\theta-\beta}{2}+\frac{c^2}{4\alpha}}\in (0,1),\\
 \varphi \in [0,4K],\\
 \begin{cases}
 \sin\frac{\theta-\beta}{2}=k\ \mathrm{sn}(\sqrt{\alpha}\varphi);\\
 \frac{c}{2}=k\sqrt{\alpha}\ \mathrm{cn}(\sqrt{\alpha}\varphi);
 \end{cases}
 \end{gather*}
 
 For $\lambda\in C_2$, we set:
 \begin{gather*}
 \alpha\neq 0, \quad E\in (\alpha,\infty),\\
 k=\sqrt{\frac{2\alpha}{E+\alpha}}=\frac{1}{\sqrt{\sin^2\left((\theta-\beta)/2\right)+c^2/(4\alpha)}}\in (0,1),\\
 \varphi \in [0,2kK],\\
 \begin{cases}
 \sin\frac{\theta-\beta}{2}=\pm k\ \mathrm{sn}\frac{\sqrt{\alpha}\varphi}{k};\\
 \frac{c}{2}=\pm \frac{\sqrt{\alpha}}{k}\ \mathrm{dn}\frac{\sqrt{\alpha}\varphi}{k};\\
 \pm=\mathrm{sgn}\ {c};
 \end{cases}
 \end{gather*}
Here $4K$ is the period of the Jacobian elliptic functions sn and cn. 

For $\lambda\in C_3$, we set
 \begin{gather*}
 \alpha\neq 0, \quad E=\alpha,\quad \theta-\beta\neq\pi,\\
 k=1,\\
 \varphi\in(-\infty,\infty),\\
 \begin{cases}
 \sin\frac{\theta-\beta}{2}=\pm \tanh(\sqrt{\alpha}\varphi);\\
 \frac{c}{2}=\pm \frac{\sqrt{\alpha}}{\cosh(\sqrt{\alpha}\varphi)};\\
 \pm=\mathrm{sgn}\ {c}.
 \end{cases}
 \end{gather*}

In the elliptic coordinates $(k,\varphi, \alpha, \beta)$ on  $\cup^3_{i=1} C_i$ the vertical part of the normal
Hamiltonian system
\[
\dot{\theta} = c,\quad \dot{c} = - \alpha \sin(\theta - \beta), \quad\dot{\alpha} = \dot{\beta} = 0,
\]
simplifies to
\[
\dot{\varphi}=1, \quad \dot{k} = \dot{\alpha} =\dot{\beta} = 0 .
\]
The elliptic coordinate $\varphi$ is the time of movement along trajectories of the pendulum equation and $k$ is a parameter that distinguishes trajectories with different energies.

\begin{proposition}
For initial covectors $\lambda$ in  $C_1$, there exists an infinite set of times of loss of equiregularity. For $\lambda\in C_2,C_3$ and $C_6$, the geodesic is equiregular for all times. Moreover, for   $\lambda\in C_4, C_5, C_7$, the corresponding geodesic is not ample, and thus abnormal.
\end{proposition}
\begin{proof}

\begin{enumerate}
\item Let $\lambda\in C_1$. Then, $c^2=4k^2\alpha\ \mathrm{cn}^2 \left(\sqrt{\alpha}\varphi_t\right)$, with $\varphi_t=\varphi+t$ and $k=\sqrt{\frac{E+1}{2}}$. In this case the geodesic has a infinite and discrete set of times of loss of equiregularity.

\item Let $\lambda\in C_2$. Then, we have $c^2=\frac{4\alpha}{k^2}\mathrm{dn}^2\ (\sqrt{\alpha}\psi(t))$,  with $\psi(t)=\frac{\varphi+t}{k}$ and $k=\sqrt{\frac{2}{E+1}}$. Hence, the geodesic is equiregular.

\item Let $\lambda\in C_3$. Then, $c^2=\frac{4\alpha}{\cosh^2(\sqrt{\alpha}\varphi_t)}$, with $\varphi_t=\varphi+t$. In this case the geodesic is also equiregular.

\item For $\lambda\in C_6$ we have $\alpha=0$, and  $c=\mathrm{const}\neq 0$. Therefore, the geodesic $\gamma$ is equiregular. 

\item For $\lambda\in C_4,C_5,C_7$, we have that $c\equiv 0$. Hence, from Proposition \ref{NonVanh3}, the corresponding geodesic is not ample, hence abnormal.
\end{enumerate}
\end{proof}

\subsection{The canonical frame}
Let $\ld$ be the initial covector associated to an ample, equiregular, unit-speed geodesic  $\gamma:\left[0,T\right]\rightarrow \mathbb{R}^5$, with $\gamma(0)=x_0$. The associated Young diagram is given by
\[
D=\begin{tabular}{ | l | l | l | l |}
  \hline
$a1$   & $a2$ & $a3$ & $a4$ \\ \hline
$b1$  & \multicolumn{1}{l}{} & \multicolumn{1}{l}{} & \multicolumn{1}{l}{} \\   \cline{1-1}
\end{tabular}
\]
 For such Young diagram, a canonical frame is a smooth family
\[
\{E_{a1},\ldots,E_{a4},E_{b1},F_{a1},\ldots,F_{a4},F_{b1}\}\in T_{\lambda_0}\left(T^*\mathbb{R}^n\right),
\]
with the following properties:
\begin{enumerate}
\item It is attached to the Jacobi curve, namely 
\[
J_{\lambda_0}(t)=\text{span}\{E_{a1}(t),E_{a2}(t),E_{a3}(t),E_{a4}(t),E_{b1}(t)\}.
\]

\item From  \cite[Lemma 5.7]{BR2} we have:
\begin{equation*}
E_{b1}(t)=e^{-t\vv{H}}_*\mathfrak{e}=\mathfrak{e}-t\vv{H},
\end{equation*}
as a consequence all curvatures $R_{\ast b,\ast 1}$, (where $\ast$ is any other index) vanish.
\item The family of symmetric matrices $R(t)$ is \emph{normal} in the sense of \cite{ZL}. In this ``easy" Young diagram case, the \emph{normal} condition means that the matrix $[R_{aa,ij}]_{i,j=1,\ldots,n_a}$ is diagonal.

\item They satisfy the structural equations:
\begin{align*}
\dot{E}_{ai}&= E_{a(i-1)}(t)  &i&=2,3,4,\\
\dot{E}_{a1}&=- F_{a1}(t) & &\\
\dot{E}_{b1}&=- F_{b1}(t) & &\\
\dot{F}_{ai}&=R_{aa,ii}(t)E_{ai}(t)-F_{a(i+1)}(t), & i&=1,2,3,\\
\dot{F}_{an_a}&=R_{aa,n_an_a}(t)E_{an_a}(t), & &\\
\dot{F}_{b1}&=0.& &
\end{align*}
\end{enumerate}
The proof of the next lemma follows the lines of the proof of Lemma \ref{IdHh}. It is, therefore, left to the interested reader.
\begin{lemma}
\label{IdHhC}
The following identities hold true:
\begin{align}
[\vv{H},X_{\theta}]&=-X_3+h_3X_{\bar{\theta}}, \label{CHX}\\
\left[\vv{H},\pa_{h_5}\right]&=-h_2\pa_{h_3}, \label{Ch5}\\
\left[\vv{H},\pa_{h_4}\right]&=-h_1\pa_{h_3}, \label{Ch4} \\
\left[\vv{H},\pa_{h_3}\right]&=-\pa_{\theta} ,\label{Ch3}\\
\left[\vv{H},\pa_{\theta}\right]&=X_{\theta}+\left(h_2h_4-h_1h_5\right)\pa_{h_3}. \label{Cht}
\end{align}
\end{lemma}
We  can start now the computation of the \emph{canonical frame}.
\begin{lemma}
$E_{a4}(t)$ is uniquely specified (up to a sign) by the following conditions
\begin{enumerate}
\item $E_{a4}(t)\in J_{\lambda_0}(t)$,
\item $E^{(i)}_{an_a}(t)\in J_{\lambda_0}(t)$, for $i=1,\ldots,3$,
\item $\sigma_{\lambda}\left(E^{(4)}_{a4}(t),E^{(3)}_{a4}(t)\right)=1$,
\end{enumerate}
and, by choosing the positive sign, is given by
\[
E_{an_a}(t)=e^{-t\vv{H}}_{\ast}\left(\frac{h_2}{h_3}\pa_{h_4}-\frac{h_1}{h_3}\pa_{h_5}\right).
\]
\end{lemma}
\begin{proof}
Condition (1) and the definition of Jacobi curve $J_{\lambda_0}(t)=e^{-t\vv{H}}_*\mathcal{V}_{\ld(t)}$ imply that
\[
E_{an_a}=e^{-t\vv{H}}_\ast\left(\sum^{5}_{i=1}f_{i}(t)\pa_{h_i}\right),
\]
for some smooth functions $f_i(t)$, with $i=1,\ldots,n$.  We compute the derivative:
\[
\dot{E}_{a4}(t)=e^{-t\vv{H}}_{\ast}\left(\sum^{5}_{i=1}f_{i}(t)[\vv{H},\pa_{h_i}]+\dot{f}_{i}(t)\pa_{h_i}\right).
\]
Condition (2) is implies  $\pi_{\ast}\circ e^{t\vv{H}}_{\ast}\dot{E}_{a4}(t)=0$. Since $\pi_{\ast}\pa_{h_i}=0$, we obtain from Lemma \ref{IdHhC},
\begin{align*}
0&=\pi_{\ast}\sum^{5}_{i=1}f_{i}(t)[\vv{H},\pa_{h_i}]\\
&=-f_{1}(t)X_1-f_2(t)X_2.
\end{align*}
From this we obtain $f_{1}=f_2\equiv0$. Then $E_{a4}(t)$ must be of the form 
\[
E_{a4}(t)=e^{-t\vv{H}}_{\ast}\left(f_3(t)\pa_{h_3}+f_4(t)\pa_{h_4}+f_5(t)\pa_{h_5}\right).
\] 
After some computations we obtain
\begin{align*}
e^{t\vv{H}}_{\ast}\dot{E}_{a4}&=-f_3\pa_{\theta}+\left(\dot{f}_3-h_1f_4-h_2f_5\right)\pa_{h_3}+\dot{f}_{4}\pa_{h_4}+\dot{f}_5\pa_{h_5},\\
e^{t\vv{H}}_{\ast}\ddot{E}_{a4}&=-f_3X_{\theta}-\left(2\dot{f}_3-h_1f_4-h_2f_5\right)\pa_{\theta}\\
&+\left(\vv{H}\left(\dot{f}_3-h_1f_4-h_2f_5\right)-f_3\left(h_2h_4-h_1h_5\right)-h_1\dot{f}_4-h_2\dot{f}_5\right)\pa_{h_3}+\ddot{f}_4\pa_{h_4}+\ddot{f}_5\pa_{h_5}.
\end{align*}
Condition (2) is equivalent to $\pi_{\ast}\circ e^{t\vv{H}}_{\ast}\ddot{E}_{a4}(t)=0$. Since $\pi_{\ast}\pa_{h_i}=0$, we obtain once again from Lemma \ref{IdHhC}
\begin{align*}
0&=\pi_{\ast}\circ e^{t\vv{H}}_{\ast}\ddot{E}_{a4}(t)=-f_3\left(h_2X_1-h_1X_2\right).
\end{align*}
This gives $f_{3}(t)\equiv 0$. Therefore, $E_{a4}(t)$ must be of the form 
\[
E_{a4}(t)=e^{-t\vv{H}}_{\ast}\left(f_4(t)\pa_{h_4}+f_5(t)\pa_{h_5}\right).
\] 
Finally,
\begin{align*}
e^{t\vv{H}}_{\ast}\dddot{E}_{a4}&=\left(h_1f_4+h_2f_5\right)X_{\theta}+\left(2\vv{H}\left(h_1f_4+h_2f_5\right)+h_1\dot{f}_4+h_2\dot{f}_5\right)\pa_{\theta}\\
&-\left(\left(h_1f_4+h_2f_5\right)\left(h_1h_5-h_2h_4\right)+\vv{H}^2\left(h_1f_4+h_2f_5\right)+\vv{H}\left(h_1\dot{f}_4+h_2\dot{f}_5\right)+h_1\ddot{f}_4+h_2\ddot{f}_5\right)\pa_{h_3}\\
&+\dddot{f}_4\pa_{h_4}+\dddot{f}_5\pa_{h_5}.
\end{align*}
Since $\pi_*\circ e^{t\vv{H}}_{\ast}\dddot{E}_{a4}=0$, we must have
\begin{equation}
\label{fg0}
h_1f_4+h_2f_5=0.
\end{equation}
Therefore,
\begin{eqnarray*}
e^{t\vv{H}}_{\ast}\dddot{E}_{a4}&=&\left(h_1\dot{f}_{4}+h_2\dot{f}_5\right)\pa_{\theta}-\left(\vv{H}\left(h_1\dot{f}_4+h_2\dot{f}_5\right)+h_1\ddot{f}_4+h_2\ddot{f}_5\right)\pa_{h_3}\\
&&+\dddot{f}_{4}\pa_{h_4}+\dddot{f}_5\pa_{h_5}.
\end{eqnarray*}
and
\begin{eqnarray*}
e^{t\vv{H}}_{\ast}E^{(4)}_{a4}&=&\left(h_1\dot{f}_4+h_2\dot{f}_5\right)X_{\theta}+\left(2\vv{H}\left(h_1\dot{f}_4+h_2\dot{f}_5\right)+h_1\ddot{f}_4+h_2\ddot{f}_5\right)\pa_{\theta}\\
&&-\left(\vv{H}^2\left(h_1\dot{f}_4+h_2\dot{f}_5\right)+\vv{H}\left(h_1\ddot{f}_4+h_2\ddot{f}_5\right)+h_1\dddot{f}_4+h_2\dddot{f}_5\right)\pa_{h_3}\\
&&+\left(h_1\dot{f}_4+h_2\dot{f}_5\right)\left(h_2h_4-h_1h_5\right)\pa_{h_3}+f^{(4)}_{4}\pa_{h_4}+f^{(4)}_5\pa_{h_5}.
\end{eqnarray*}
Then, 
\begin{eqnarray*}
1=\sigma_{\lambda}\left(E^{(4)}_{a4}(t),\dddot{E}_{a4}(t)\right)&=&\sigma_{\lambda(t)}\left(e^{t\vv{H}}_{\ast}E^{(4)}_{a4}(t),e^{t\vv{H}}_{\ast}\dddot{E}_{a4}(t)\right)\\
&=&\left(h_1\dot{f}_4(t)+h_2\dot{f}_5(t)\right)^2.
\end{eqnarray*}
By choosing the positive sign, using equation (\ref{fg0}), and the geodesic equations, we have that
\begin{equation}
\label{fg1}
1=h_1\dot{f}_{4}+h_2\dot{f}_5=h_3\left(h_2f_4-h_1f_5\right).
\end{equation}
The solution to the system of equations given by (\ref{fg0}) and (\ref{fg1}) is 
\[
f_4=\frac{h_2}{h_3},\quad f_5=-\frac{h_1}{h_3}.
\]
\end{proof}

If we take the derivative of Eq. (\ref{fg1}), we obtain
\begin{eqnarray*}
h_1\ddot{f}_4+h_2\ddot{f}_5&=&-\dot{h}_1\dot{f}_4-\dot{h}_2\dot{f}_5\\
&=& h_3\left(h_2\dot{f}_4-h_1\dot{f}_5\right).
\end{eqnarray*}
A straightforward computation gives us
 \begin{eqnarray*}
\dot{f}_4&=&h_1-\frac{h_2}{h^2_3}\left(h_1h_4+h_2h_5\right),\\
\dot{f}_5&=&h_2+\frac{h_1}{h^2_3}\left(h_1h_4+h_2h_5\right).
\end{eqnarray*}
Therefore
\[
h_1\ddot{f}_4+h_2\ddot{f}_5=-\frac{h_1h_4+h_2h_5}{h_3}.
\]
Furthermore, we have
\[
h_1\dddot{f}_4+h_2\dddot{f}_5=\vv{H}\left(h_1\ddot{f}_4+h_2\ddot{f}_5\right)-\dot{h}_1\ddot{f}_4-\dot{h}_2\ddot{f}_5.
\]
After some computations, we have
\begin{eqnarray*}
\ddot{f}_4&=&-h_2h_3-\frac{1}{h^3_3}\left(h^2_3\left(2h_1h_2h_5+h_4\left(h^2_1-h^2_2\right)\right)-2h_2\left(h_1h_4+h_2h_5\right)^2\right),\\
\ddot{f}_5&=&h_1h_3+\frac{1}{h^3_3}\left(h^2_3\left(-2h_1h_2h_4+h_5\left(h^2_1-h^2_2\right)\right)-2h_1\left(h_1h_4+h_2h_5\right)^2\right),
\end{eqnarray*}
and 
\[
\vv{H}\left(h_1\ddot{f}_4+h_2\ddot{f}_5\right)=h_2h_4-h_1h_5+\frac{\left(h_1h_4+h_2h_5\right)^2}{h^2_3}.
\]
Hence,
\begin{equation}
\dot{h}_1\ddot{f}_4+\dot{h}_2\ddot{f}_5=\left(h^2_3+\left(h_1h_5-h_2h_4\right)-\frac{2}{h^2_3}\left(h_1h_4+h_2h_5\right)^2\right).
\end{equation}
Therefore,
\[
h_1\dddot{f}_4+h_2\dddot{f}_5=-\left(h^2_3+2\left(h_1h_5-h_2h_4\right)-\frac{3}{h^2_3}\left(h_1h_4+h_2h_5\right)^2\right).
\]
Finally, we obtain the following expressions for $F_{a1}$ and $\dot{F}_{a1}$:
\begin{eqnarray}
\nonumber
F_{a1}(0)&=&-\left(h_1\dot{f}_4+h_2\dot{f}_5\right)X_{\theta}-\left(h_1\ddot{f}_4+h_2\ddot{f}_5\right)\pa_{\theta}+\left(\vv{H}\left(h_1\ddot{f}_4+h_2\ddot{f}_5\right)+h_1\dddot{f}_4+h_2\dddot{f}_5\right)\pa_{h_3}\\ \nonumber
&&-\left(h_1\dot{f}_4+h_2\dot{f}_5\right)\left(h_2h_4-h_1h_5\right)\pa_{h_3}-f^{(4)}_{4}\pa_{h_4}-f^{(4)}_5\pa_{h_5}\\ \nonumber
&=&-X_{\theta}+\frac{h_1h_4+h_2h_5}{h_3}\partial_{\theta}-\left(h^2_3+2\left(h_1h_5-h_2h_4\right)-\frac{4}{h^2_3}\left(h_1h_4+h_2h_5\right)^2\right)\partial_{h_3}\\ \label{CFa1}
&&\mod\mathcal{V}_{4},
\end{eqnarray}
and
\begin{eqnarray}
\nonumber
\dot{F}_{a1}(0)&=&X_3-h_3X_{\bar{\theta}}+\frac{h_1h_4+h_2h_5}{h_3}X_{\theta}+\left(h^2_3+3\left(h_1h_5-h_2h_4\right)-\frac{5}{h^2_3}\left(h_1h_4+h_2h_5\right)^2\right)\partial_{\theta}\\ \label{CdFa1}
&& \mod \mathcal{V}_{3},
\end{eqnarray}
where $\mathcal{V}_{3}=\mathrm{span}\ \{\partial_{h_3},\partial_{h_4},\partial_{h_5}\}$ and $\mathcal{V}_{4}=\mathrm{span}\ \{\partial_{h_4},\partial_{h_5}\}$.
\begin{proof} [Proof of Theorem \ref{casym}] If we use formulas (\ref{CFa1}) and (\ref{CdFa1}), we obtain
\begin{equation*}
R_{aa,11}=\sigma_{\lambda_0}\left(\dot{F}_{a1}(0),F_{a1}(0)\right)=3h^2_3+6\left(h_1h_5-h_2h_4\right)-\frac{8}{h^2_3}\left(h_1h_4+h_2h_5\right)^2.
\end{equation*}
In terms of the \emph{first integral}
\[
E=\frac{h^2_3}{2}+h_1h_5-h_2h_4,
\]
and the coordinates $\left(\theta,c,\alpha,\beta\right)$ introduced in Section 6.1, we can write
\begin{eqnarray*}
R_{aa,11}&=&6E-8\frac{\alpha^2}{c^2}\sin^2(\theta-\beta).\\
&\leq& 6E.
\end{eqnarray*}
By direct inspection, the orthonormal basis $\{X_a,X_b\}$ for $\mathscr{D}_{x_0}$ obtained by the projection of the canonical frame is
 \[
 X_a\doteq\pi_{\ast}F_{a1}(0), \quad  X_b\doteq\pi_{\ast}F_{b1}(0). 
 \]
In the coordinates associated to the splitting $\Sigma_{\lambda}=\mathcal{V}_{\lambda_0}\oplus\mathcal{H}_{\lambda_0}$ we have 
\[
\mathcal{Q}_{\lambda_0}(t)=\frac{d}{dt}S^{\flat}(t)^{-1},
\]
where $S^{\flat}(t)^{-1}$, in the basis $\{X_a,X_b\}$, is given by:
\[
S^{\flat}(t)^{-1}=-\frac{1}{t}
\begin{pmatrix}
16& 0\\
0& 1
\end{pmatrix}
+\frac{4}{63}
\begin{pmatrix}
R_{aa,11} & 0\\
0 & 0
\end{pmatrix}
t+O\left(t^2\right).
\]
Therefore, the curvature operator has the following expression
\begin{equation*}
\mathcal{R}_{\lambda_0}=
\begin{pmatrix}
\frac{4}{21}R_{aa,11}(0)&0\\
0&0
\end{pmatrix},
\end{equation*}
where $R_{aa,11}=6E-8\frac{\alpha^2}{c^2}\sin^2(\theta-\beta).$
 \end{proof}


\begin{thebibliography}{99}

\bibitem{ABB1}
	{{Agrachev}, A.,  {Barilari}, D. and {Boscain}, U.}
	 {Introduction to {R}iemannian and sub-{R}iemannian geometry (From Hamiltonian viewpoint)}.
	 {Available at \url { http://www.math.jussieu.fr/~barilari/Notes.phpl } (2016)},
	
\bibitem{ABR1}
	{{Agrachev}, A., {Barilari}, D. and {Rizzi}, L.}
	{{The curvature: a variational approach}}.
	\emph{To appear in Memoirs of the AMS}. 2015.
	 
\bibitem{ABR2}
	{{Agrachev}, A.,  {Barilari}, D. and {Rizzi}, L.}
	{Sub-{R}iemannian curvature in contact geometry}.
	\emph{The Journal of Geometric Analysis}. {2016};
	 {1--43}.
\bibitem{AG}
	{Agrachev, A. and Gamkrelidze, R. V.}
	{Feedback-invariant optimal control theory and differential geometry. {I}. {R}egular extremals}.
	\emph{J. Dynam. Control Systems}.
	{1997};
	{{3}:343--389},
	 
	
	
\bibitem{AL}
        {Agrachev, A. and Lee, P.}
        {{Generalized Ricci curvature bounds for three dimensional contact sub-Riemannian manifolds}},
	\emph{Math. Ann.}
	{2014};
	{360};
	 {1-2}:
	 {209--253}.
	 
\bibitem{AL1}
	{Agrachev, A. and Lee, P.}
	{{Bishop and Laplacian comparison theorems on three-dimensional contact sub-Riemannian manifolds with symmetry}}.
	\emph{J. Geom. Anal.},
	2015;
	25(1)
	 {512--535}.
	
	
\bibitem{AZ}
	 {Agrachev, A. and Zelenko, I.}
	 {Geometry of {J}acobi curves. {I}},
	\emph{J. Dynam. Control Systems}.
	{2002};
	8(1):
	 {93--140}.
	

	
	
\bibitem{AnzaldoMonroy}
	{Anzaldo-Meneses, A. and Monroy-P{\'e}rez, F.}
	{{Goursat distribution and sub-Riemannian structures}},
	\emph{Journal of Mathematical Physics}.
	{2003};
	44:
	{6101-6111}.
	
	

\bibitem{ArSac}
	 {Ardentov, A.A. and Sachkov, Yu. L.}
	 {{Extremal trajectories in a nilpotent sub-Riemannian problem on the Engel group}}.
	 \emph{Sbornik: Mathematics}.
	 {2011};
	202({11}): 31-54.
	
	
	
\bibitem{BR1}
	{{Barilari}, D. and {Rizzi}, L.}
	{{Comparison theorems for conjugate points in sub-Riemannian geometry}}.
	\emph{ESAIM: COCV}. 22 2 (2016) 439-472
        doi: http://dx.doi.org/10.1051/cocv/2015013
	
\bibitem{BR2}
	{{Barilari}, D. and {Rizzi}, L.}
	{{On Jacobi fields and canonical connection in sub-Riemannian geometry}}.
	{ArXiv e-prints}.
	{2015}.
	
\bibitem{FBo} F. Baudoin and M. Bonnefont. Curvature-dimension estimates for the Laplace-Beltrami operator of a totally geodesic foliation. Nonlinear Anal., 2015;126:159-169.

\bibitem{BBG} F. Baudoin, M. Bonnefont, and N. Garofalo. A sub-Riemannian curvature-dimension inequality, volume doubling property and the Poincar\'e inequality. \emph{Math. Ann.} 2014; 3-4:833-860.

\bibitem{BBGM} F. Baudoin, M. Bonnefont, N. Garofalo, and I. Munive. Volume and distance comparison theorems for sub-Riemannian manifolds. \emph{Journal of Functional Analysis}. 2014; 267(7): 2005-2027

\bibitem{BG1} F. Baudoin and N. Garofalo. Curvature-dimension inequalities and Ricci lower bounds for sub-Riemannian manifolds with transverse symmetries. \emph{To appear in Journal of the EMS}.

\bibitem{BG2} F. Baudoin and N. Garofalo. A note on boundedness of Riesz transform for some subelliptic operators.  \emph{International Mathematics Research Notices.} 2013; 2:398-421.

\bibitem{FK2} F. Baudoin and B. Kim. The Lichnerowicz-Obata theorem on sub-Riemannian manifolds with transverse symmetries. \emph{To appear in Journal of Geometric Analysis}. 2014.

\bibitem{FK1} F. Baudoin and B. Kim. Sobolev, Poincar\'e and isoperimetric inequalities for subelliptic diffusion operators satisfying a generalized curvature dimension inequality.\emph{Revista Matematica Iberoamericana.} 2014; 30:109-131. 

\bibitem{BKW} F. Baudoin, B. Kim, and J. Wang. Transverse Weitzenb\"ock formulas and curvature dimension inequalities on Riemannian foliations with totally geodesic leaves. \emph{To appear in Comm. Anal. Geom. 2016}.

\bibitem{BWa} F. Baudoin and J. Wang. Curvature dimension inequalities and subelliptic heat kernel gradient bounds on contact manifolds. \emph{Potential Anal.} 2014; 40:163-193.

\bibitem{Cor} J.-M. Coron. Control and nonlinearity, \emph{Mathematical Surveys and Monographs}, vol. 136, American Mathematical Society, Providence, RI, 2007. 

\bibitem{FLMR} M. Fliess, J. L\'evine, P. Martin, and P. Rouchon. Flatness and defect of nonlinear systems: Introductory theory and examples, \emph{International Journal of Control}. 1995; 61:1327-1361.

\bibitem{FLMR1} M. Fliess, J. L\'evine,  P. Martin and P. Rouchon. Flatness and motion planning: the car with $n$  trailers. In \emph{Proc. of the European Control Conference.} pp 1518-1522, Groningen, The Netherlands, 1993.

\bibitem{Goursat} E. Goursat. Lecons sur le probleme de Pfaff. \emph{Hermann, Paris.} 1923.

\bibitem{Jak} B. Jakubzcyk. Invariants of dynamic feedback and free systems.  \emph{Proc. of the European Control Conference.} pp 1510-1513. Groningen, The Netherlands, 1993.

\bibitem{JeanF} F. Jean. The car with $n$ trailers: Characterization of the singular configurations, \emph{ESAIM: COCV.} 1996; 1:241-266.

\bibitem{EssaysRobotics} B. John, S. Sastry, and H. Sussmann, editors. Essays on Mathematical Robotics. \emph{The IMA
Volumes in Mathematics and its Applications}. Vol. 104; Springer-Verlag New York, 1998.

\bibitem{PBG}  V. G. Boltyanskii, R. V. Gamkrelidze, R. V. G. L. S. Pontryagin. The mathematical theory of optimal processes. Translated from
the Russian by K. N. Trirogoff; edited by L. W. Neustadt. \emph{Interscience Publishers John Wiley \& Sons, Inc.}.
New York-London, 1962.

\bibitem{Laumond} J.-P. Laumond. Controllability of a multibody mobile robot. \emph{Robotics and Automation, IEEE Transactions
on.} 1993; 9(6): 755-763.

\bibitem{Laumond2} J.-P. Laumond. Robot Motion Planning and Control, \emph{Lecture Notes on Control and Information Sciences.}
Springer-Verlag, Berlin, 1997.

\bibitem{Li-Canny} X. Li and J. Canny, editors. Nonholonomic Motion Planning.\emph{The Springer International Series in Engineering and Computer Science}. Vol. 192;  Springer US, 1993.

 \bibitem{Loeper} G. Loeper. On the regularity of solutions of optimal transportation problems. \emph{Acta Math.} 2009; 202(2):241-283.

\bibitem{Mon} R. Montgomery. A tour of sub-Riemannian geometries, their geodesics and applications.  \emph{Mathematical Surveys and Monographs,} American Mathematical Society. Vol. 91; Providence, RI, 2002.

\bibitem{Ri} L. Rizzi. Measure contraction properties of Carnot groups. \emph{Calc. Var.} 2016; 55: 60. 

\bibitem{Sac1} Y. L. Sachkov. An exponential mapping in the generalized Dido problem. \emph{Mat. Sb.,} 2003; 194(9): 63-90.

\bibitem{Samson} C. Samson. Control of chain systems: Application to path following and time-varying point stabilization of
mobile robots. \emph{IEEE Trans. Automat. Control.} 1995; 40: 64-77.

\bibitem{Sordalen} O. Sordalen. Conversion of the kinematics of a car with $n$ trailers into a chained form. In \emph{Proc. of the IEEE
Conference on Robotics and Automation.} pp. 382-387, Atlanta, Georgia, 1993.

\bibitem{TMW} A. Teel, R. Murray, and G. Walsh. Nonholonomic control systems: From steering to stabilization with
sinusoids, \emph{International Journal of Control}, 1995; 62: 849-870.

\bibitem{TMS} D. Tilbury, R. Murray, and S. Sastry. Trajectory generation for the $n$-trailer problem using Goursat normal
form. \emph{IEEE Trans. Automat. Control.} 1995; 802-819.

\bibitem{WhW} E. Whittaker and G. Watson. A course on modern analysis. \emph{Cambridge Univ. Press}. New York, 1962.

\bibitem{ZL} I. Zelenko and C. Li. Differential geometry of curves in Lagrange Grassmannians with given Young diagram.
\emph{Differential Geom. Appl.} 2009; 27(6): 723-742.
	

\end{thebibliography}
\end{document}